\documentclass{amsart}

\usepackage{amssymb,amsmath,amsthm,latexsym,booktabs,amscd}

\theoremstyle{definition}
\newtheorem{definition}{Definition}[section]
\newtheorem{remark}[definition]{Remark}
\newtheorem{example}[definition]{Example}

\theoremstyle{plain}
\newtheorem{lemma}[definition]{Lemma}
\newtheorem{theorem}[definition]{Theorem}

\newtheorem{conjecture}[definition]{Conjecture}
\newtheorem{problem}[definition]{Problem}

\begin{document}

\title{Algebras, Dialgebras, and Polynomial Identities}

\author{Murray R. Bremner}

\address{Department of Mathematics and Statistics, University of Saskatchewan, Canada}

\email{bremner@math.usask.ca}

\dedicatory{Dedicated to Yuri Bahturin on his 65th birthday}



\thanks{This paper is an expanded version of the lecture notes from the author's talk at
the Second International Workshop on Polynomial Identities, 2--6 September 2011,
which took place at the Atlantic Algebra Centre, Memorial University, St.~John's,
Newfoundland, Canada.}

\maketitle

\begin{abstract}
This is a survey of some recent developments in the theory of associative and nonassociative dialgebras,
with an emphasis on polynomial identities and multilinear operations.
We discuss associative, Lie, Jordan, and alternative algebras, and the corresponding dialgebras;
the KP algorithm for converting identities for algebras into identities for dialgebras;
the BSO algorithm for converting operations in algebras into operations in dialgebras; Lie and
Jordan triple systems, and the corresponding disystems; and a noncommutative version of Lie triple
systems based on the trilinear operation $abc-bca$.  The paper concludes with
a conjecture relating the KP and BSO algorithms, and some suggestions for further research.
Most of the original results are joint work with Ra\'ul Felipe,
Luiz A. Peresi, and Juana S\'anchez-Ortega.
\end{abstract}


\section{Algebras}

Throughout this talk the base field $\mathbb{F}$ will be arbitrary, but we usually exclude low characteristics,
especially $p \le n$ where $n$ is the degree of the polynomial identities under consideration.
The assumption $p > n$ allows us to assume that all polynomial identities are multilinear and that the group
algebra $\mathbb{F} S_n$ is semisimple.

\begin{definition}
An \textbf{algebra} is a vector space $A$ with a bilinear operation
  \[
  \mu\colon A \times A \to A.
  \]
Unless otherwise specified, we write $ab = \mu(a,b)$ for $a, b \in A$. We say
that $A$ is \textbf{associative} if it satisfies the polynomial identity
  \[
  (ab)c \equiv a(bc).
  \]
Throughout this paper we will use the symbol $\equiv$ to indicate an equation that holds for all values of the
arguments; in this case, all $a, b, c \in A$.
\end{definition}

\begin{theorem}
The free unital associative algebra on a set $X$ of generators has basis consisting of all words of degree $n \ge 0$,
  \[
  x = x_1 x_2 \cdots x_n,
  \quad \text{where} \quad
  x_1, x_2, \dots, x_n \in X,
  \]
with the product defined on basis elements by concatenation and extended bilinearly,
  \[
  ( x_1 x_2 \cdots x_m ) ( y_1 y_2 \cdots y_n )
  =
  x_1 x_2 \cdots x_m y_1 y_2 \cdots y_n.
  \]
\end{theorem}

\begin{definition}
The \textbf{commutator} in an algebra is the bilinear operation
  \[
  [a,b] = ab - ba.
  \]
This operation is \textbf{anticommutative}: it satisfies $[a,b] + [b,a] \equiv 0$.
\end{definition}

\begin{lemma}
In an associative algebra, the commutator satisfies the identity
  \[
  [[a,b],c] + [[b,c],a] + [[c,a],b] \equiv 0
  \qquad
  \emph{(Jacobi)}
  \]
\end{lemma}

\begin{definition}
A \textbf{Lie algebra} is an algebra which satisfies anticommutativity and the Jacobi identity.
\end{definition}

\begin{theorem} \emph{\textbf{Poincar\'e-Birkhoff-Witt.}}
Every Lie algebra $L$ has a universal associative enveloping algebra $U(L)$ for which the canonical map $L \to U(L)$
is injective. It follows that every polynomial identity satisfied by the commutator in every associative algebra is
a consequence of anticommutativity and the Jacobi identity.
\end{theorem}

\begin{remark}
Most texts on Lie algebras include a proof of the PBW Theorem.
The most beautiful proof is that of Bergman \cite{Bergman} using noncommutative Gr\"obner bases;
see also de Graaf \cite[Ch.~6]{deGraaf}.
For the history of the PBW Theorem, see Grivel \cite{Grivel}.
For a survey on Gr\"obner-Shirshov bases, see Bokut and Kolesnikov \cite{BokutKolesnikov}.
\end{remark}

\begin{definition}
The \textbf{anticommutator} in an algebra is the bilinear operation
  \[
  a \circ b = ab + ba;
  \]
we omit the scalar $\frac12$.
This operation is \textbf{commutative}: it satisfies $a \circ b - b \circ a \equiv 0$.
\end{definition}

\begin{lemma}
In an associative algebra, the anticommutator satisfies the identity
  \[
  ( ( a \circ a ) \circ b ) \circ a - ( a \circ a ) \circ ( b \circ a ) \equiv 0
  \qquad
  \emph{(Jordan)}
  \]
\end{lemma}

\begin{definition}
A \textbf{Jordan algebra} is an algebra which satisfies commutativity and the Jordan identity.
\end{definition}

\begin{theorem}
There exist polynomial identities satisfied by the anticommutator in every associative algebra which do not follow
from commutativity and the Jordan identity.  The lowest degree in which such identities exist is 8.
\end{theorem}

\begin{remark}
A Jordan algebra is called \emph{special} if it is isomorphic to a subspace of an associative algebra
closed under the anticommutator.
An polynomial identity for Jordan algebras is called \emph{special} if it is satisfied by all special
Jordan algebras but not by all Jordan algebras.
The first special identities for Jordan algebras were found by Glennie \cite{Glennie1,Glennie2}.
For a computational approach, see Hentzel \cite{Hentzel}.
Another s-identity was obtained by Thedy \cite{Thedy}; see also McCrimmon \cite{McCrimmon1}
and \cite[Appendix B.5]{McCrimmonBook}.
For a survey on identities in Jordan algebras, see McCrimmon \cite{McCrimmon2}.
\end{remark}

\begin{remark}
From the perspective of polynomial identities, there is a clear dichotomy between the two bilinear operations,
commutator and anticommutator.  Both operations satisfy simple identities in low degree; for the commutator, these
identities imply all the identities satisfied by the operation, but for the anticommutator, there exist special
identities of higher degree.
\end{remark}


\section{Dialgebras}

We now recall the concept of a dialgebra: a vector space with two multiplications.
Associative dialgebras were originally defined by Loday in the 1990s, and the results
quoted in this section were proved by him; see especially his original paper \cite{Loday1}
and his survey article \cite{Loday2}.  Associative dialgebras provide the natural setting
for Leibniz algebras, a ``non-anticommutative'' generalization of Lie algebras; see Loday
\cite{Loday3}.

\begin{definition}
A \textbf{dialgebra} is a vector space $A$ with two bilinear operations,
  \[
  \dashv \colon A \times A \to A,
  \qquad
  \vdash \colon A \times A \to A,
  \]
called the \textbf{left} and \textbf{right} products.
We say that $A$ is a \textbf{0-dialgebra} if it satisfies the \textbf{left} and \textbf{right bar identities},
  \[
  ( a \dashv b ) \vdash c \equiv ( a \vdash b ) \vdash c,
  \qquad
  a \dashv ( b \dashv c ) \equiv a \dashv ( b \vdash c ).
  \]
An \textbf{associative dialgebra} is a 0-dialgebra satisfying \textbf{left}, \textbf{right},
and \textbf{inner associativity}:
  \[
  ( a \dashv b ) \dashv c \equiv a \dashv ( b \dashv c ),
  \quad
  ( a \vdash b ) \vdash c \equiv a \vdash ( b \vdash c ),
  \quad
  ( a \vdash b ) \dashv c \equiv a \vdash ( b \dashv c ).
  \]
\end{definition}

\begin{definition}
Let $x = x_1 x_2 \cdots x_n$ be a monomial in an associative dialgebra,
with some placement of parentheses and choice of operations.  The \textbf{center}
of $x$, denoted $c(x)$, is defined by induction on $n$:
  \begin{itemize}
  \item
If $n = 1$ then $x = x_1$ and $c(x) = x_1$.
  \item
If $n \ge 2$ then $x = y \dashv z$ or $x = y \vdash z$,
and $c(x) = c(y)$ or $c(x) = c(z)$ respectively.
  \end{itemize}
\end{definition}

\begin{lemma} \label{lemmanormalform}
Let $x = x_1 x_2 \cdots x_n$ be a monomial in an associative dialgebra with $c(x) = x_i$. Then
the following expression does not depend on the placement of parentheses:
  \[
  x = x_1 \vdash \cdots \vdash x_{i-1} \vdash x_i \dashv x_{i+1} \dashv \cdots \dashv x_n.
  \]
\end{lemma}

\begin{definition}
The expression in the Lemma \ref{lemmanormalform} is called the \textbf{normal form} of the monomial $x$,
and is abbreviated using the \textbf{hat notation}:
  \[
  x = x_1 \cdots \widehat{x}_i \cdots x_n.
  \]
\end{definition}

\begin{theorem}
The free associative dialgebra on a set $X$ of
generators has basis consisting of all monomials in normal form:
  \[
  x = x_1 \cdots \widehat{x}_i \cdots x_n
  \qquad
  (1 \le i \le n, \; x_1, x_2, \dots, x_n \in X).
  \]
Two such monomials are equal if and only if they have the same permutation of the generators and
the same position of the center.
The left and right products are defined on monomials as follows and extended bilinearly:
  \allowdisplaybreaks
  \begin{align*}
  x \dashv y
  &=
  ( x_1 \cdots \widehat{x}_i \cdots x_n ) \dashv ( y_1 \cdots \widehat{y}_j \cdots y_p )
  =
  x_1 \cdots \widehat{x}_i \cdots x_n y_1 \cdots y_p,
  \\
  x \vdash y
  &=
  ( x_1 \cdots \widehat{x}_i \cdots x_n ) \vdash ( y_1 \cdots \widehat{y}_j \cdots y_p )
  =
  x_1 \cdots x_n y_1 \cdots \widehat{y}_j \cdots y_p.
  \end{align*}
\end{theorem}

\begin{definition}
The \textbf{dicommutator} in a dialgebra is the bilinear operation
  \[
  \langle a, b \rangle = a \dashv b - b \vdash a.
  \]
In general, this operation is not anticommutative.
\end{definition}

\begin{lemma}
In an associative dialgebra, the dicommutator satisfies the identity
  \[
  \langle \langle a, b \rangle, c \rangle
  \equiv
  \langle \langle a, c \rangle, b \rangle
  +
  \langle a, \langle b, c \rangle \rangle
  \qquad
  \emph{(Leibniz)}
  \]
\end{lemma}

\begin{definition}
A \textbf{Leibniz algebra} (or Lie dialgebra) is an algebra satisfying the Leibniz identity.
\end{definition}

\begin{remark} \label{Leibnizremark}
If we set $b = c$ in the Leibniz identity then we obtain $\langle a, \langle b, b \rangle \rangle \equiv 0$,
and the linearized form of this identity (assuming characteristic not 2) is
  \[
  \langle a, \langle b, c \rangle \rangle + \langle a, \langle c, b \rangle \rangle \equiv 0
  \qquad
  \text{(right anticommutativity)}
  \]
\end{remark}

\begin{theorem} \emph{\textbf{Loday-Pirashvili.}}
Every Leibniz algebra $L$ has a universal associative
enveloping dialgebra $U(L)$ for which the canonical map $L \to U(L)$ is
injective. Hence every polynomial identity satisfied by the
dicommutator in every associative dialgebra is a consequence of the Leibniz identity.
\end{theorem}

\begin{remark}
The Loday-Pirashvili Theorem is the generalization to dialgebras of the PBW Theorem.
For the original proof, see \cite{LodayPirashvili}.
For different approaches, see Aymon and Grivel \cite{AymonGrivel}, Insua and Ladra \cite{InsuaLadra}.
\end{remark}

\begin{remark}
The definition of associative dialgebra can be motivated in terms of the Leibniz identity.
If we expand the Leibniz identity in a nonassociative dialgebra using the dicommutator as
the operation, then we obtain
  \allowdisplaybreaks
  \begin{align*}
  &
  ( a \dashv b - b \vdash a ) \dashv c - c \vdash ( a \dashv b - b \vdash a )
  \equiv
  \\
  &
  ( a \dashv c - c \vdash a ) \dashv b - b \vdash ( a \dashv c - c \vdash a )
  +
  a \dashv ( b \dashv c - c \vdash b ) - ( b \dashv c - c \vdash b ) \vdash a.
  \end{align*}
Equating terms with the same permutation of $a, b, c$ gives the following identities:
  \allowdisplaybreaks
  \begin{alignat*}{2}
  &( a \dashv b ) \dashv c \equiv a \dashv ( b \dashv c ),
  &\qquad
  &0 \equiv ( a \dashv c ) \dashv b - a \dashv ( c \vdash b ),
  \\
  &( b \vdash a ) \dashv c \equiv b \vdash ( a \dashv c ),
  &\qquad
  &0 \equiv b \vdash ( c \vdash a ) - ( b \dashv c ) \vdash a,
  \\
  &c \vdash ( a \dashv b ) \equiv ( c \vdash a ) \dashv b,
  &\qquad
  &c \vdash ( b \vdash a ) \equiv ( c \vdash b ) \vdash a.
  \end{alignat*}
These are equivalent to the identities defining associative dialgebras.
\end{remark}

\begin{definition}
The \textbf{antidicommutator} in a dialgebra is the bilinear operation
  \[
  a \star b = a \dashv b + b \vdash a.
  \]
In general, this operation is not commutative.
\end{definition}

\begin{lemma} \label{lemmajordandialgebra}
In an associative dialgebra, the antidicommutator satisfies
  \allowdisplaybreaks
  \begin{alignat*}{2}
  a \star ( b \star c ) &\equiv a \star ( c \star b )
  &\qquad
  &\emph{(right commutativity)}
  \\
  ( b  \star a^2 ) \star a &\equiv ( b \star a ) \star a^2
  &\qquad
  &\emph{(right Jordan identity)}
  \\
  \langle a, b, c^2 \rangle &\equiv 2 \langle a \star c, b, c \rangle
  &\qquad
  &\emph{(right Osborn identity)}
  \end{alignat*}
where $a^2 = a \star a$ and
$\langle a, b, c \rangle = ( a \star b ) \star c - a \star ( b \star c )$.
\end{lemma}

\begin{remark}
These identities were obtained independently by different authors:
Vel\'asquez and Felipe \cite{VelasquezFelipe}, Kolesnikov \cite{Kolesnikov}, Bremner \cite{Bremner}.
A generalization of the TKK construction from Lie and Jordan algebras to Lie and Jordan dialgebras
has been given by Gubarev and Kolesnikov \cite{GubarevKolesnikov}.
For further work on the structure of Jordan dialgebras, see Felipe \cite{Felipe}.
I have named the last identity in Lemma \ref{lemmajordandialgebra} after Osborn \cite{Osborn};
it is a noncommutative version of the identity stating that a commutator of multiplications is a derivation.
\end{remark}

\begin{definition}
A \textbf{Jordan dialgebra} (or quasi-Jordan algebra) is an algebra satisfying
right commutativity and the right Jordan and Osborn identities.
\end{definition}

\begin{remark}
Strictly speaking, Leibniz algebras and Jordan dialgebras have two operations, but they are opposite,
so we consider only one.
This will become clear when we discuss the KP algorithm for converting
identities for algebras into identities for dialgebras.
\end{remark}

\begin{theorem}
There exist special identities for Jordan dialgebras; that is,
polynomial identities satisfied by the antidicommutator in every associative dialgebra
which are not consequences of right commutativity and the right Jordan and Osborn identities.
\end{theorem}

\begin{remark}
This result was obtained using computer algebra by Bremner and Peresi \cite{BremnerPeresi}.
The lowest degree for such identities is 8;
some but not all of these identities are noncommutative versions of the Glennie identity.
For a theoretical approach to similar results, including generalizations of the classical theorems
of Cohn, Macdonald, and Shirshov, see Voronin \cite{Voronin}.
\end{remark}


\section{From algebras to dialgebras}

We now discuss a general approach to the following problem.

\begin{problem}
Given a polynomial identity for algebras,
how do we obtain the corresponding polynomial identity (or identities) for dialgebras?
\end{problem}

An algorithm has been developed by Kolesnikov and Pozhidaev for converting multilinear
identities for algebras into multilinear identities for dialgebras.
For binary algebras, see \cite{Kolesnikov};
for the generalization to $n$-ary algebras, see \cite{Pozhidaev}.
The underlying structure from the theory of operads is discussed by Chapoton \cite{Chapoton}.

\subsection*{Kolesnikov-Pozhidaev (KP) algorithm}

The input is a multilinear polynomial identity of degree $d$ for
an $n$-ary operation denoted by the symbol $\{-,\cdots,-\}$ with $n$ arguments.
The output of Part 1 is a collection of $d$ multilinear identities of degree $d$
for $n$ new $n$-ary operations denoted $\{-,\cdots,-\}_i$ for $1 \le i \le n$.
The output of Part 2 is a collection of multilinear identities of degree $2n{-}1$
for the same new operations.

\subsubsection*{Part 1}

Given a multilinear identity of degree $d$ in the $n$-ary operation $\{-,\cdots,-\}$,
we describe the application of the algorithm to one monomial, and
extend this by linearity to the entire identity. Let $a_1 a_2 \dots a_d$ be a
multilinear monomial of degree $d$ with some placement of $n$-ary operation
symbols. For each $i = 1, \dots, d$ we convert the monomial $a_1 a_2 \dots a_d$
in the original $n$-ary operation into a new monomial of the same degree in the
$n$ new $n$-ary operations, according to the following rule, based on the
position of the variable $a_i$, called the central variable of the
monomial. For each occurrence of the original $n$-ary operation in the
monomial, either $a_i$ occurs in one of the $n$ arguments or not, and we
have two cases:
  \begin{enumerate}
  \item[(a)]
  If $a_i$ occurs in the $j$-th argument then we convert $\{-,\cdots,-\}$ to
  the $j$-th new operation symbol $\{-,\cdots,-\}_j$.
  \item[(b)]
  If $a_i$ does not occur in any of the $n$ arguments, then either
    \begin{itemize}
    \item
    $a_i$ occurs to the left of $\{-,\cdots,-\}$:
    we convert $\{-,\cdots,-\}$ to the first new operation symbol $\{-,\cdots,-\}_1$, or
    \item
    $a_i$ occurs to the right of $\{-,\cdots,-\}$:
    we convert $\{-,\cdots,-\}$ to the last new operation symbol $\{-,\cdots,-\}_n$.
    \end{itemize}
  \end{enumerate}

\subsubsection*{Part 2}

We also include the following identities,
generalizing the bar identities for associative dialgebras, for all $i, j =
1, \dots, n$ with $i \ne j$ and all $k, \ell = 1, \dots, n$:
  \allowdisplaybreaks
  \begin{align*}
  &
  \{ a_1, \dots, a_{i-1}, \{ b_1, \cdots, b_n \}_k, a_{i+1}, \dots, a_n \}_j
  \equiv
  \\
  &
  \{ a_1, \dots, a_{i-1}, \{ b_1, \cdots, b_n \}_\ell, a_{i+1}, \dots, a_n \}_j.
  \end{align*}
This identity says that the $n$ new operations are interchangeable in
the $i$-th argument of the $j$-th new operation when $i \ne j$.

\begin{example}
The definition of associative dialgebra can be obtained by applying
the KP algorithm to the associativity identity, which we write in the form
  \[
  \{ \{ a, b \}, c \} \equiv \{ a, \{ b, c \} \}.
  \]
The operation $\{-,-\}$ produces two new operations $\{-,-\}_1$, $\{-,-\}_2$.
Part 1 of the algorithm produces three identities by making $a, b, c$ in turn the central variable:
  \allowdisplaybreaks
  \begin{alignat*}{2}
  &
  \{ \{ a, b \}_1, c \}_1 \equiv \{ a, \{ b, c \}_1 \}_1,
  &
  \qquad
  &
  \{ \{ a, b \}_2, c \}_1 \equiv \{ a, \{ b, c \}_1 \}_2,
  \\
  &
  \{ \{ a, b \}_2, c \}_2 \equiv \{ a, \{ b, c \}_2 \}_2.
  \end{alignat*}
Part 2 of the algorithm produces two identities:
  \[
  \{ a, \{ b, c \}_1 \}_1 \equiv \{ a, \{ b, c \}_2 \}_1,
  \qquad
  \{ \{ a, b \}_1, c \}_2 \equiv \{ \{ a, b \}_2, c \}_2.
  \]
If we write $a \dashv b$ for $\{ a, b \}_1$ and $a \vdash b$ for $\{ a, b \}_2$
then these are the three associativity identities and the two bar identities.
\end{example}

\begin{example}
The definition of Leibniz algebra
can be obtained by applying the KP algorithm to
the identities defining Lie algebras:
anticommutativity (in its bilinear form) and the Jacobi identity,
  \[
  [a,b] + [b,a] \equiv 0,
  \qquad
  [[a,b],c] + [[b,c],a] + [[c,a],b] \equiv 0.
  \]
Part 1 of the algorithm produces five identities:
  \allowdisplaybreaks
  \begin{alignat*}{2}
  &
  [a,b]_1 + [b,a]_2 \equiv 0,
  &\qquad
  &
  [[a,b]_1,c]_1 + [[b,c]_2,a]_2 + [[c,a]_2,b]_1 \equiv 0,
  \\
  &
  [a,b]_2 + [b,a]_1 \equiv 0,
  &\qquad
  &
  [[a,b]_2,c]_1 + [[b,c]_1,a]_1 + [[c,a]_2,b]_2 \equiv 0,
  \\
  &&&
  [[a,b]_2,c]_2 + [[b,c]_2,a]_1 + [[c,a]_1,b]_1 \equiv 0.
  \end{alignat*}
The two identities of degree 2 are equivalent to $[a,b]_2 \equiv -[b,a]_1$,
so the second operation is superfluous.
Eliminating the second operation from the three identities of degree 3 shows that each of them is
equivalent to the identity
  \[
  [[a,b]_1,c]_1 + [a,[c,b]_1]_1 - [[a,c]_1,b]_1 \equiv 0.
  \]
If we write $\langle a, b \rangle = [a,b]_1$ then we obtain a form of the Leibniz identity.
Part 2 of the algorithm produces two identities:
  \[
  [ a, [ b, c ]_1 ]_1 \equiv [ a, [ b, c ]_2 ]_1,
  \qquad
  [ [ a, b ]_1, c ]_2 \equiv [ [ a, b ]_2, c ]_2.
  \]
Eliminating the second operation gives right anticommutativity:
  \[
  \langle a, \langle b, c \rangle \rangle
  +
  \langle a, \langle c, b \rangle \rangle
  \equiv
  0.
  \]
However, as we have already seen in Remark \ref{Leibnizremark}, the Leibniz identity implies right anticommutativity,
so it suffices to retain only the Leibniz identity.
\end{example}

\begin{example}
To apply the KP algorithm to the defining identities for Jordan algebras, we write commutativity and
the multilinear form of the Jordan identity using the operation symbol $\{-,-\}$:
  \allowdisplaybreaks
  \begin{align*}
  &
  \{ a, b \} - \{ b, a \} \equiv 0,
  \\
  &
  \{ \{ \{ a, c \}, b \}, d \}
  +
  \{ \{ \{ a, d \}, b \}, c \}
  +
  \{ \{ \{ c, d \}, b \}, a \}
  \\
  &\quad
  -
  \{ \{ a, c \}, \{ b, d \} \}
  -
  \{ \{ a, d \}, \{ b, c \} \}
  -
  \{ \{ c, d \}, \{ b, a \} \}
  \equiv 0.
  \end{align*}
From commutativity, Part 1 of the algorithm gives two identities of degree 2:
\[
\{ a, b \}_1 - \{ b, a \}_2 \equiv 0,
\qquad
\{ a, b \}_2 - \{ b, a \}_1 \equiv 0,
\]
These two identities are equivalent to $\{ a, b \}_2 \equiv \{ b, a \}_1$:
the second operation is the opposite of the first, and so we may eliminate $\{-,-\}_2$.
From the linearized Jordan identity, Part 1 of the algorithm gives four identities of degree 4:
  \allowdisplaybreaks
\begin{align*}
&
\{ \{ \{ a, c \}_1, b \}_1, d \}_1
+
\{ \{ \{ a, d \}_1, b \}_1, c \}_1
+
\{ \{ \{ c, d \}_2, b \}_2, a \}_2
\\
&
-
\{ \{ a, c \}_1, \{ b, d \}_1 \}_1
-
\{ \{ a, d \}_1, \{ b, c \}_1 \}_1
-
\{ \{ c, d \}_2, \{ b, a \}_2 \}_2
\equiv 0,
\\
&
\{ \{ \{ a, c \}_2, b \}_2, d \}_1
+
\{ \{ \{ a, d \}_2, b \}_2, c \}_1
+
\{ \{ \{ c, d \}_2, b \}_2, a \}_1
\\
&
-
\{ \{ a, c \}_2, \{ b, d \}_1 \}_2
-
\{ \{ a, d \}_2, \{ b, c \}_1 \}_2
-
\{ \{ c, d \}_2, \{ b, a \}_1 \}_2
\equiv 0,
\\
&
\{ \{ \{ a, c \}_2, b \}_1, d \}_1
+
\{ \{ \{ a, d \}_2, b \}_2, c \}_2
+
\{ \{ \{ c, d \}_1, b \}_1, a \}_1
\\
&
-
\{ \{ a, c \}_2, \{ b, d \}_1 \}_1
-
\{ \{ a, d \}_2, \{ b, c \}_2 \}_2
-
\{ \{ c, d \}_1, \{ b, a \}_1 \}_1
\equiv 0,
\\
&
\{ \{ \{ a, c \}_2, b \}_2, d \}_2
+
\{ \{ \{ a, d \}_2, b \}_1, c \}_1
+
\{ \{ \{ c, d \}_2, b \}_1, a \}_1
\\
&
-
\{ \{ a, c \}_2, \{ b, d \}_2 \}_2
-
\{ \{ a, d \}_2, \{ b, c \}_1 \}_1
-
\{ \{ c, d \}_2, \{ b, a \}_1 \}_1
\equiv 0.
\end{align*}
We replace every instance of $\{-,-\}_2$ by the opposite of $\{-,-\}_1$:
  \allowdisplaybreaks
\begin{align*}
&
\{ \{ \{ a, c \}_1, b \}_1, d \}_1
+
\{ \{ \{ a, d \}_1, b \}_1, c \}_1
+
\{ a, \{ b, \{ d, c \}_1 \}_1 \}_1
\\
&
-
\{ \{ a, c \}_1, \{ b, d \}_1 \}_1
-
\{ \{ a, d \}_1, \{ b, c \}_1 \}_1
-
\{ \{ a, b \}_1, \{ d, c \}_1 \}_1
\equiv 0,
\\
&
\{ \{ b, \{ c, a \}_1 \}_1, d \}_1
+
\{ \{ b, \{ d, a \}_1 \}_1, c \}_1
+
\{ \{ b, \{ d, c \}_1 \}_1, a \}_1
\\
&
-
\{ \{ b, d \}_1, \{ c, a \}_1 \}_1
-
\{ \{ b, c \}_1, \{ d, a \}_1 \}_1
-
\{ \{ b, a \}_1, \{ d, c \}_1 \}_1
\equiv 0,
\\
&
\{ \{ \{ c, a \}_1, b \}_1, d \}_1
+
\{ c, \{ b, \{ d, a \}_1 \}_1 \}_1
+
\{ \{ \{ c, d \}_1, b \}_1, a \}_1
\\
&
-
\{ \{ c, a \}_1, \{ b, d \}_1 \}_1
-
\{ \{ c, b \}_1, \{ d, a \}_1 \}_1
-
\{ \{ c, d \}_1, \{ b, a \}_1 \}_1
\equiv 0,
\\
&
\{ d, \{ b, \{ c, a \}_1 \}_1 \}_1
+
\{ \{ \{ d, a \}_1, b \}_1, c \}_1
+
\{ \{ \{ d, c \}_1, b \}_1, a \}_1
\\
&
-
\{ \{ d, b \}_1, \{ c, a \}_1 \}_1
-
\{ \{ d, a \}_1, \{ b, c \}_1 \}_1
-
\{ \{ d, c \}_1, \{ b, a \}_1 \}_1
\equiv 0.
\end{align*}
We simplify the notation and write $\{a,b\}_1$ as $ab$.
The last four identities become:
\allowdisplaybreaks
\begin{align*}
&
( ( a c ) b ) d
+
( ( a d ) b ) c
+
a ( b ( d c ) )
-
( a c ) ( b d )
-
( a d ) ( b c )
-
( a b ) ( d c )
\equiv 0,
\\
&
( b ( c a ) ) d
+
( b ( d a ) ) c
+
( b ( d c ) ) a
-
( b d ) ( c a )
-
( b c ) ( d a )
-
( b a ) ( d c )
\equiv 0,
\\
&
( ( c a ) b ) d
+
c ( b ( d a ) )
+
( ( c d ) b ) a
-
( c a ) ( b d )
-
( c b ) ( d a )
-
( c d ) ( b a )
\equiv 0,
\\
&
d ( b ( c a ) )
+
( ( d a ) b ) c
+
( ( d c ) b ) a
-
( d b ) ( c a )
-
( d a ) ( b c )
-
( d c ) ( b a )
\equiv 0.
\end{align*}
The first is equivalent to the third and to the fourth, so we retain only the first and second.
Part 2 of the algorithm produces two identities:
\[
\{ a, \{ b, c \}_1 \}_1 \equiv \{ a, \{ b, c \}_2 \}_1,
\qquad
\{ \{ a, b \}_1, c \}_2 \equiv \{ \{ a, b \}_2, c \}_2.
\]
Rewriting these using only the first operation gives
\[
\{ a, \{ b, c \}_1 \}_1 \equiv \{ a, \{ c, b \}_1 \}_1,
\qquad
\{ c, \{ a, b \}_1 \}_1 \equiv \{ c, \{ b, a \}_1 \}_1.
\]
These two identities are equivalent to right commutativity: $a(bc) \equiv a(cb)$.
We rearrange the two retained identities of degree 4
and apply right commutativity:
  \allowdisplaybreaks
\begin{align*}
&
( ( a c ) b ) d
-
( a c ) ( b d )
+
( ( a d ) b ) c
-
( a d ) ( b c )
-
( a b ) ( c d )
+
a ( b ( c d ) )
\equiv 0,
\\
&
( b ( a c ) ) d
+
( b ( a d ) ) c
+
( b ( c d ) ) a
-
( b d ) ( a c )
-
( b c ) ( a d )
-
( b a ) ( c d )
\equiv 0.
\end{align*}
The first identity can be reformulated in terms of associators as follows,
\[
( a c, b, d )
+
( a d, b, c )
-
( a, b, c d )
\equiv 0,
\]
and assuming characteristic $\ne 2$ this is equivalent to
\[
( a, b, c^2 ) \equiv 2 ( a c, b, c ).
\]
Setting $a = c = d$ in the second identity and assuming characteristic $\ne 3$ gives
\[
( b a^2 ) a \equiv ( b a ) a^2,
\]
Thus we obtain right commutativity and the right Osborn and Jordan identities.
\end{example}

\begin{example}
The multilinear forms of the left and right alternative identities defining alternative algebras are:
  \[
  (a,b,c) + (b,a,c) \equiv 0,
  \qquad
  (a,b,c) + (a,c,b) \equiv 0.
  \]
Expanding the associators gives
  \[
  (ab)c - a(bc) + (ba)c - b(ac) \equiv 0,
  \qquad
  (ab)c - a(bc) + (ac)b - a(cb) \equiv 0.
  \]
We apply the KP algorithm to these identities, writing $\{-,-\}$ for the original bilinear operation.
Part 1 gives six identities
relating the two new operations $\{-,-\}_1$ and $\{-,-\}_2$: in each of the two
original identities we make either $a$, $b$, or $c$ the central argument.
In this case, we retain both operations, since there is no identity of degree 2
relating $\{-,-\}_1$ and $\{-,-\}_2$.
We obtain six identities defining alternative dialgebras; in the first (second) group of three,
the only differences are in the subscripts 1 and 2 indicating the position of the central variable:
  \allowdisplaybreaks
  \begin{align*}
  \{ \{ a, b \}_1, c \}_1 - \{ a, \{ b, c \}_1 \}_1 + \{ \{ b, a \}_2, c \}_1 - \{ b, \{ a, c \}_1 \}_2 &\equiv 0,
  \\
  \{ \{ a, b \}_2, c \}_1 - \{ a, \{ b, c \}_1 \}_2 + \{ \{ b, a \}_1, c \}_1 - \{ b, \{ a, c \}_1 \}_1 &\equiv 0,
  \\
  \{ \{ a, b \}_2, c \}_2 - \{ a, \{ b, c \}_2 \}_2 + \{ \{ b, a \}_2, c \}_2 - \{ b, \{ a, c \}_2 \}_2 &\equiv 0,
  \\
  \{ \{ a, b \}_1, c \}_1 - \{ a, \{ b, c \}_1 \}_1 + \{ \{ a, c \}_1, b \}_1 - \{ a, \{ c, b \}_1 \}_1 &\equiv 0,
  \\
  \{ \{ a, b \}_2, c \}_1 - \{ a, \{ b, c \}_1 \}_2 + \{ \{ a, c \}_2, b \}_2 - \{ a, \{ c, b \}_2 \}_2 &\equiv 0,
  \\
  \{ \{ a, b \}_2, c \}_2 - \{ a, \{ b, c \}_2 \}_2 + \{ \{ a, c \}_2, b \}_1 - \{ a, \{ c, b \}_1 \}_2 &\equiv 0.
  \end{align*}
We revert to standard notation: $\dashv$ for $\{-,-\}_1$ and $\vdash$ for $\{-,-\}_2$:
  \allowdisplaybreaks
  \begin{align*}
  ( a \dashv b ) \dashv c - a \dashv ( b \dashv c ) + ( b \vdash a ) \dashv c - b \vdash ( a \dashv c ) &\equiv 0,
  \\
  ( a \vdash b ) \dashv c - a \vdash ( b \dashv c ) + ( b \dashv a ) \dashv c - b \dashv ( a \dashv c ) &\equiv 0,
  \\
  ( a \vdash b ) \vdash c - a \vdash ( b \vdash c ) + ( b \vdash a ) \vdash c - b \vdash ( a \vdash c ) &\equiv 0,
  \\
  ( a \dashv b ) \dashv c - a \dashv ( b \dashv c ) + ( a \dashv c ) \dashv b - a \dashv ( c \dashv b ) &\equiv 0,
  \\
  ( a \vdash b ) \dashv c - a \vdash ( b \dashv c ) + ( a \vdash c ) \vdash b - a \vdash ( c \vdash b ) &\equiv 0,
  \\
  ( a \vdash b ) \vdash c - a \vdash ( b \vdash c ) + ( a \vdash c ) \dashv b - a \vdash ( c \dashv b ) &\equiv 0.
  \end{align*}
We rewrite these in terms of the left, right and inner associators:
  \allowdisplaybreaks
  \begin{alignat*}{2}
  &
  ( a, b, c )_\dashv + ( b, a, c )_\times \equiv 0,
  &\qquad
  &
  ( a, b, c )_\times + ( b, a, c )_\dashv \equiv 0,
  \\
  &
  ( a, b, c )_\vdash + ( b, a, c )_\vdash \equiv 0,
  &\qquad
  &
  ( a, b, c )_\dashv + ( a, c, b )_\dashv \equiv 0,
  \\
  &
  ( a, b, c )_\times + ( a, c, b )_\vdash \equiv 0,
  &\qquad
  &
  ( a, b, c )_\vdash + ( a, c, b )_\times \equiv 0.
  \end{alignat*}
These six identities show how the associators change under various transpositions of the arguments.
In particular, the identities in the second row show that
the right operation $a \vdash b$ is left alternative, and
the left operation $a \dashv b$ is right alternative.
(We do not have two alternative operations.)
Part 2 of the algorithm simply gives the left and right bar identities.
To summarize, we define an alternative dialgebra to be a 0-dialgebra satisfying
  \[
  (a,b,c)_\dashv + (c,b,a)_\vdash \equiv 0,
  \quad
  (a,b,c)_\dashv - (b,c,a)_\vdash \equiv 0,
  \quad
  (a,b,c)_\times + (a,c,b)_\vdash \equiv 0,
  \]
where the left, right, and inner associators are defined by
  \allowdisplaybreaks
  \begin{alignat*}{2}
  (a,b,c)_\dashv
  &=
  (a \dashv b) \dashv c - a \dashv (b \dashv c),
  &\qquad
  (a,b,c)_\vdash
  &=
  (a \vdash b) \vdash c - a \vdash (b \vdash c),
  \\
  (a,b,c)_\times
  &=
  (a \vdash b) \dashv c - a \vdash (b \dashv c).
  \end{alignat*}
This definition was originally obtained in a different way by Liu \cite{Liu}.
\end{example}

\begin{example}
Malcev algebras \cite{PSMalcev} can be defined by the polynomial identities of degree
$\le 4$ satisfied by the commutator in every alternative algebra.
Bremner, Peresi and S\'anchez-Ortega \cite{BPSO} used computer algebra to study
the identities satisfied by the dicommutator in every alternative dialgebra,
and proved that every such identity of degree $\le 6$ is a consequence of the identities
of degree $\le 4$.  They showed that the
identities of degree $\le 4$ are equivalent to those obtained by applying the KP algorithm
to linearized forms of anticommutativity the Malcev identity, namely right anticommutativity
and a ``noncommutative'' version of the Malcev identity:
  \[
  a(bc) + a(cb) \equiv 0,	
  \qquad
  ((ab)c)d - ((ad)b)c - (a(cd))b - (ac)(bd) - a((bc)d) \equiv 0.
  \]
These two identities define the variety of Malcev dialgebras.
\end{example}


\section{Multilinear operations}

We now consider generalizations of the commutator $ab-ba$ and anticommutator $ab+ba$
to operations of arbitrary ``arity'' (number of arguments).
The following definitions and examples are based primarily on Bremner and Peresi \cite{BremnerPeresi2}.

\begin{definition}
A \textbf{multilinear $n$-ary operation} $\omega( a_1, a_2, \dots, a_n )$,
or more concisely an $n$-linear operation,
is a linear combination of permutations of the monomial
$a_1 a_2 \cdots a_n$ regarded as an element of the free associative algebra on $n$ generators:
  \[
  \omega( a_1, a_2, \dots, a_n )
  =
  \sum_{\sigma \in S_n} x_\sigma \, a_{\sigma(1)} a_{\sigma(2)} \cdots a_{\sigma(n)}
  \qquad
  ( x_\sigma \in \mathbb{F} ).
  \]
We identify $\omega( a_1, a_2, \dots, a_n )$ with an element of $\mathbb{F} S_n$,
the group algebra of the symmetric group $S_n$ which acts by permuting the subscripts
of the generators.
\end{definition}

\begin{definition}
Two multilinear operations are \textbf{equivalent} if each is a linear combination of
permutations of the other; this is the same as saying that the
two operations generate the same left ideal in $\mathbb{F} S_n$.
\end{definition}

\begin{example}
For $n = 2$, we have the Wedderburn decomposition
$\mathbb{F} S_2 \approx \mathbb{F} \oplus \mathbb{F}$,
where the two simple ideals correspond to partitions 2 and $1+1$ and have bases $ab+ba$ and $ab-ba$
respectively (writing $a, b$ instead of $a_1, a_2$).
There are four equivalence classes, corresponding to the commutator, the
anticommutator, the zero operation, and the original associative operation $ab$.
\end{example}

\begin{table}
\[
\begin{array}{rll}
&\qquad
\text{reduced matrix form}
&\qquad
\text{permutation form}
\\
\toprule
1
&\qquad
\left[ \, 0, \, \left[\begin{array}{rr} 0 & 1 \\ 0 & 0 \end{array}\right], \, 0 \, \right]
&\qquad
abc-bac-cab+cba
\\[10pt]
2
&\qquad
\left[ \, 0, \, \left[\begin{array}{rr} 1 & 1/2 \\ 0 & 0 \end{array}\right], \, 0 \, \right]
&\qquad
abc+acb-bca-cba
\\[10pt]
3
&\qquad
\left[ \, 1, \, \left[\begin{array}{rr} 0 & 1 \\ 0 & 0 \end{array}\right], \, 0 \, \right]
&\qquad
abc+cba
\\[10pt]
4
&\qquad
\left[ \, 1, \, \left[\begin{array}{rr} 1 & 0 \\ 0 & 0 \end{array}\right], \, 0 \, \right]
&\qquad
abc+bac
\\[10pt]
5
&\qquad
\left[ \, 1, \, \left[\begin{array}{rr} 1 & 1 \\ 0 & 0 \end{array}\right], \, 0 \, \right]
&\qquad
abc+acb
\\[10pt]
6
&\qquad
\left[ \, 1, \, \left[\begin{array}{rr} 1 & 1/2 \\ 0 & 0 \end{array}\right], \, 0 \, \right]
&\qquad
2abc+acb+2bac+bca
\\[10pt]
7
&\qquad
\left[ \, 0, \, \left[\begin{array}{rr} 0 & 1 \\ 0 & 0 \end{array}\right], \, 1 \, \right]
&\qquad
2abc-acb-2bac+bca
\\[10pt]
8
&\qquad
\left[ \, 0, \, \left[\begin{array}{rr} 1 & -1 \\ 0 & 0 \end{array}\right], \, 1 \, \right]
&\qquad
abc-acb
\\[10pt]
9
&\qquad
\left[ \, 0, \, \left[\begin{array}{rr} 1 & 2 \\ 0 & 0 \end{array}\right], \, 1 \, \right]
&\qquad
abc-bac
\\[10pt]
10
&\qquad
\left[ \, 0, \, \left[\begin{array}{rr} 1 & 1/2 \\ 0 & 0 \end{array}\right], \, 1 \, \right]
&\qquad
abc-cba
\\[10pt]
11
&\qquad
\left[ \, 1, \, \left[\begin{array}{rr} 0 & 1 \\ 0 & 0 \end{array}\right], \, 1 \, \right]
&\qquad
abc-bac+bca
\\[10pt]
12
&\qquad
\left[ \, 1, \, \left[\begin{array}{rr} 1 & 0 \\ 0 & 0 \end{array}\right], \, 1 \, \right]
&\qquad
abc+cab-cba
\\[10pt]
13
&\qquad
\left[ \, 1, \, \left[\begin{array}{rr} 1 & 1 \\ 0 & 0 \end{array}\right], \, 1 \, \right]
&\qquad
abc+bca-cba
\\[10pt]
14
&\qquad
\left[ \, 1, \, \left[\begin{array}{rr} 1 & -1 \\ 0 & 0 \end{array}\right], \, 1 \, \right]
&\qquad
abc+bac+cab
\\[10pt]
15
&\qquad
\left[ \, 1, \, \left[\begin{array}{rr} 1 & 2 \\ 0 & 0 \end{array}\right], \, 1 \, \right]
&\qquad
abc+acb+bca
\\[10pt]
16
&\qquad
\left[ \, 1, \, \left[\begin{array}{rr} 1 & 1/2 \\ 0 & 0 \end{array}\right], \, 1 \, \right]
&\qquad
abc+acb+bac
\\[10pt]
17
&\qquad
\left[ \, 0, \, \left[\begin{array}{rr} 1 & 0 \\ 0 & 1 \end{array}\right], \, 0 \, \right]
&\qquad
abc-bca
\\[10pt]
18
&\qquad
\left[ \, 1, \, \left[\begin{array}{rr} 1 & 0 \\ 0 & 1 \end{array}\right], \, 0 \, \right]
&\qquad
abc+acb+bac-cba
\\[10pt]
19
&\qquad
\left[ \, 0, \, \left[\begin{array}{rr} 1 & 0 \\ 0 & 1 \end{array}\right], \, 1 \, \right]
&\qquad
abc+acb-bca-cab
\\
\bottomrule
\end{array}
\]
\medskip
\caption{Simplified trilinear operations}
\label{trilinearoperations}
\end{table}

\begin{example}
For $n = 3$, we have the Wedderburn decomposition
  \[
  \mathbb{F} S_3 \approx \mathbb{F} \oplus M_2(\mathbb{F}) \oplus \mathbb{F},
  \]
where the simple ideals correspond to partitions 3, $2+1$ and $1+1+1$.
As representatives of the equivalence classes of trilinear operations we take ordered triples of matrices in row canonical form:
  \[
  \left[ \;
  x, \;
  \begin{bmatrix} y_{11} & y_{12} \\ y_{21} & y_{22} \end{bmatrix}, \;
  z \;
  \right]
  \]
The first and third components are either 0 or 1; the second can be one of
  \[
  \begin{bmatrix} 0 & 0 \\ 0 & 0 \end{bmatrix},
  \qquad
  \begin{bmatrix} 1 & q \\ 0 & 0 \end{bmatrix} \; (q \in \mathbb{F}),
  \qquad
  \begin{bmatrix} 0 & 1 \\ 0 & 0 \end{bmatrix},
  \qquad
  \begin{bmatrix} 1 & 0 \\ 0 & 1 \end{bmatrix}.
  \]
There are infinitely many equivalence classes:
four infinite families (for which the $2 \times 2$ matrix has rank 1)
and six isolated operations (for which the $2 \times 2$ matrix has rank 0 or 2).
In order to classify these operations, we consider two bases for the group algebra $\mathbb{F} S_3$,
assuming that the characteristic of $\mathbb{F}$ is not 2 or 3.
The first basis consists of the permutations in lexicographical order:
  \[
  abc, \quad
  acb, \quad
  bac, \quad
  bca, \quad
  cab, \quad
  cba.
  \]
The second basis consists of the matrix units for the Wedderburn decomposition:
  \allowdisplaybreaks
  \begin{align*}
  S &= \tfrac16 ( abc + acb + bca + bca + cab + cba ),
  \\
  E_{11} &= \tfrac13 ( abc + bca - bca - cba ),
  \qquad
  E_{12} = \tfrac13 ( acb - bca + bca - cab ),
  \\
  E_{21} &= \tfrac13 ( acb - bca + cab - cba ),
  \qquad
  E_{22} = \tfrac13 ( abc - bca - cab + cba ),
  \\
  A &= \tfrac16 ( abc - acb - bca + bca + cab - cba ).
  \end{align*}
The change of basis matrices are
  \[
  M =
  \frac16
  \left[
  \begin{array}{rrrrrr}
  1 &\!\!  2 &\!\!  0 &\!\!  0 &\!\!  2 &\!\!  1 \\
  1 &\!\!  0 &\!\!  2 &\!\!  2 &\!\!  0 &\!\! -1 \\
  1 &\!\!  2 &\!\! -2 &\!\!  0 &\!\! -2 &\!\! -1 \\
  1 &\!\! -2 &\!\!  2 &\!\! -2 &\!\!  0 &\!\!  1 \\
  1 &\!\!  0 &\!\! -2 &\!\!  2 &\!\! -2 &\!\!  1 \\
  1 &\!\! -2 &\!\!  0 &\!\! -2 &\!\!  2 &\!\! -1
  \end{array}
  \right],
  \quad
  M^{-1}
  =
  \left[
  \begin{array}{rrrrrr}
  1 &\!\!  1 &\!\!  1 &\!\!  1 &\!\!  1 &\!\!  1 \\
  1 &\!\!  0 &\!\!  1 &\!\!  0 &\!\! -1 &\!\! -1 \\
  0 &\!\!  1 &\!\!  0 &\!\!  1 &\!\! -1 &\!\! -1 \\
  0 &\!\!  1 &\!\! -1 &\!\! -1 &\!\!  1 &\!\!  0 \\
  1 &\!\!  0 &\!\! -1 &\!\! -1 &\!\!  0 &\!\!  1 \\
  1 &\!\! -1 &\!\! -1 &\!\!  1 &\!\!  1 &\!\! -1
  \end{array}
  \right]
  \]
Except for the associative operation $abc$, all these operations satisfy polynomial identities in degree 3.
Bremner and Peresi \cite{BremnerPeresi2} identified 19 of these operations which satisfy polynomial identities in degree 5
which do not follow from the identities in degree 3.
These operations are given in Table \ref{trilinearoperations}, which contains the representative of
the equivalence class in matrix form and the simplest operation in that class written as a linear combination of
permutations.
(The simplified forms of the operations were found by enumerating all $5^6 = 15625$ linear combinations of the
permutations with coefficients $\{ 0, \pm 1, \pm 2 \}$, computing the reduced matrix form of
each of the resulting group algebra elements, and recording those which belong to the same equivalence class
as one of the operations from \cite{BremnerPeresi2}.)
This list includes the Lie and anti-Lie triple products,
  \[
  abc - bac - cab + cba,
  \qquad
  abc + bac - cab - cba,
  \]
and the Jordan and anti-Jordan triple products,
  \[
  abc + cba,
  \qquad
  abc - cba.
  \]
The list does not include the symmetric, alternating, and cyclic sums,
  \allowdisplaybreaks
  \begin{align*}
  &
  abc + acb + bca + bca + cab + cba,
  \qquad
  abc - acb - bca + bca + cab - cba,
  \\
  &
  abc + bca + cab,
  \end{align*}
since every polynomial identity of degree 5 satisfied by these operations is a consequence of the identities in degree 3.
In other words, there are no new identities until degree 7;
see Bremner and Hentzel \cite{BremnerHentzel}.
\end{example}

We now discuss a general approach to the following problem.

\begin{problem}
Given a multilinear operation for algebras, how do we obtain the corresponding operation (or operations) for dialgebras?
\end{problem}

A simple algorithm which converts a multilinear operation of degree $n$ in an associative algebra
into $n$ multilinear operations of degree $n$ in an associative dialgebra was introduced by
Bremner and S\'anchez-Ortega \cite{BSO}.

\subsection*{Bremner--S\'anchez-Ortega (BSO) algorithm}

The input is a multilinear $n$-ary operation $\omega$ in an associative algebra:
\[
\omega( a_1, a_2, \dots, a_n )
=
\sum_{\sigma \in S_n}
x_\sigma \, a_{\sigma(1)} a_{\sigma(2)} \cdots a_{\sigma(n)}
\qquad
(x_\sigma \in \mathbb{F}).
\]
For each $i = 1, 2, \dots, n$ we partition the set of all permutations into subsets
according to the position of $i$:
  \[
  S_n^{j,i}
  =
  \{ \, \sigma \in S_n \mid \sigma(j) = i \, \}.
  \]
For each $i = 1, 2, \dots, n$ we collect the terms of $\omega$ in which $a_i$ is in position $j$:
\[
\omega_i( a_1, a_2, \dots, a_n )
=
\sum_{j=1}^n
\sum_{S_n^{j,i}}
x_\sigma \, a_{\sigma(1)} \cdots a_{\sigma(j-1)} a_i a_{\sigma(j+1)} \cdots a_{\sigma(n)}.
\]
The output consists of $n$ new multilinear $n$-ary operations
$\widehat{\omega}_1$, \dots, $\widehat{\omega}_n$
in an associative dialgebra, obtained from $\omega$ by making $a_i$ the center of each term:
  \[
  \widehat{\omega}_i( a_1, a_2, \dots, a_n )
  =
  \sum_{j=1}^n
  \sum_{S_n^{(i)}}
  x_\sigma \, a_{\sigma(1)} \cdots a_{\sigma(j-1)} \, \widehat{a}_i \,
  a_{\sigma(j+1)} \cdots a_{\sigma(n)}.
  \]

\begin{example}
The commutator $ab - ba$ produces two dicommutators;
the second is the negative of the opposite of the first, $\langle a, b \rangle_2 = - \langle b, a \rangle_1$:
  \[
  \langle a, b \rangle_1 = \widehat{a} b - b \widehat{a},
  \qquad
  \langle a, b \rangle_2 = a \widehat{b} - \widehat{b} a.
  \]
The anticommutator $ab + ba$ produces two antidicommutators;
the second is the opposite of the first,
$\langle a, b \rangle_2 = \langle b, a \rangle_1$:
  \[
  \langle a, b \rangle_1 = \widehat{a} b + b \widehat{a},
  \qquad
  \langle a, b \rangle_2 = a \widehat{b} + \widehat{b} a.
  \]
\end{example}

\begin{example} \label{ltsexample}
We apply the BSO algorithm to the Lie triple product,
  \[
  \omega(a,b,c) = abc - bac - cab + cba.
  \]
We obtain these three dialgebra operations:
  \allowdisplaybreaks
  \begin{alignat*}{2}
  \widehat{\omega}_1(a,b,c) &= \widehat{a}bc - b\widehat{a}c - c\widehat{a}b + cb\widehat{a},
  &\qquad
  \widehat{\omega}_2(a,b,c) &= a\widehat{b}c - \widehat{b}ac - ca\widehat{b} + c\widehat{b}a,
  \\
  \widehat{\omega}_3(a,b,c) &= ab\widehat{c} - ba\widehat{c} - \widehat{c}ab + \widehat{c}ba.
  \end{alignat*}
We have
  \[
  \widehat{\omega}_2(a,b,c) = - \widehat{\omega}_1(b,a,c),
  \qquad
  \widehat{\omega}_3(a,b,c) = \widehat{\omega}_1(c,b,a) - \widehat{\omega}_1(c,a,b),
  \]
so we only retain $\widehat{\omega}_1(a,b,c)$.
\end{example}

\begin{example} \label{jtsexample}
We apply the BSO algorithm to the Jordan triple product,
  \[
  \omega(a,b,c) = abc + cba.
  \]
We obtain these three dialgebra operations:
  \[
  \widehat{\omega}_1(a,b,c) = \widehat{a}bc + cb\widehat{a},
  \qquad
  \widehat{\omega}_2(a,b,c) = a\widehat{b}c + c\widehat{b}a,
  \qquad
  \widehat{\omega}_3(a,b,c) = ab\widehat{c} + \widehat{c}ba.
  \]
We have $\widehat{\omega}_3(a,b,c) = \widehat{\omega}_1(c,b,a)$, so we only retain
$\widehat{\omega}_1(a,b,c)$ and $\widehat{\omega}_2(a,b,c)$.
The second operation is symmetric in its first and third arguments:
$\widehat{\omega}_2(c,b,a) = \widehat{\omega}_2(a,b,c)$.
\end{example}


\section{Leibniz triple systems}

We consider the dialgebra analogue of Lie triple systems.
We apply the KP algorithm to the defining polynomial identities, and then find the identities
satisfied by the operations obtained from the BSO algorithm applied to the Lie triple product.
We then use computer algebra to verify that the results are equivalent.
This section is a summary of Bremner and S\'anchez-Ortega \cite{BSO2}.
We assume that the base field $\mathbb{F}$ does not have characteristic 2, 3 or 5.

\begin{definition} \label{definitionLTS}
A \textbf{Lie triple system} is a vector space $T$ with a trilinear operation $T \times T \times T \to T$
denoted $[a,b,c]$ satisfying these multilinear identities:
  \allowdisplaybreaks
  \begin{align*}
  &[a,b,c] + [b,a,c] \equiv 0,
  \\
  &[a,b,c] + [b,c,a] + [c,a,b] \equiv 0,
  \\
  &[a,b,[c,d,e]] - [[a,b,c],d,e] - [c,[a,b,d],e] - [c,d,[a,b,e]] \equiv 0.
  \end{align*}
These identities are satisfied by the Lie triple product in any associative algebra.
\end{definition}

\subsection{KP algorithm}

Applying Part 1 of the algorithm to the identities of degree 3 in Definition \ref{definitionLTS} gives
  \allowdisplaybreaks
  \begin{alignat*}{2}
  &[a,b,c]_1 + [b,a,c]_2 \equiv 0,
  &\qquad
  &[a,b,c]_1 + [b,c,a]_3 + [c,a,b]_2 \equiv 0,
  \\
  &[a,b,c]_2 + [b,a,c]_1 \equiv 0,
  &\qquad
  &[a,b,c]_2 + [b,c,a]_1 + [c,a,b]_3 \equiv 0,
  \\
  &[a,b,c]_3 + [b,a,c]_3 \equiv 0,
  &\qquad
  &[a,b,c]_3 + [b,c,a]_2 + [c,a,b]_1 \equiv 0.
  \end{alignat*}
The first two identities on the left are equivalent and show that $[-,-,-]_2$ is superfluous;
the three on the right are equivalent and show that $[-,-,-]_3$ is superfluous:
  \[
  [a,b,c]_2 \equiv - [b,a,c]_1,
  \qquad
  [a,b,c]_3 \equiv - [b,c,a]_2 - [c,a,b]_1 \equiv [c,b,a]_1 - [c,a,b]_1.
  \]
We retain only the first operation which we write as $\langle -,-,- \rangle$.
Applying Part 1 of the algorithm to the identity of degree 5 in Definition \ref{definitionLTS} gives five identities, two of which are redundant.
We use the previous equations to eliminate $[-,-,-]_2$ and $[-,-,-]_3$ from the remaining three identities and obtain:
  \begin{equation} \label{lts1}
  \left\{
  \begin{array}{l}
  \langle a, b, \langle c, d, e \rangle \rangle
  -
  \langle \langle a, b, c \rangle, d, e \rangle
  +
  \langle \langle a, b, d \rangle, c, e \rangle
  -
  \langle \langle a, b, e \rangle, d, c \rangle
  \\[2pt]
  \;\;
  {} +
  \langle \langle a, b, e\rangle , c, d\rangle
  \equiv 0,
  \\[2pt]
  \langle \langle c, d, e \rangle, b, a \rangle
  -
  \langle \langle c, d, e \rangle, a, b \rangle
  -
  \langle \langle c, b, a \rangle, d, e \rangle
  +
  \langle \langle c, a, b \rangle, d, e \rangle
  \\[2pt]
  \;\;
  {} -
  \langle c, \langle a, b, d \rangle, e \rangle
  -
  \langle c, d, \langle a, b, e \rangle \rangle
  \equiv 0,
  \\[2pt]
  \langle \langle e, d, c \rangle, b, a \rangle
  -
  \langle \langle e, c, d \rangle, b, a \rangle
  -
  \langle \langle e, d, c \rangle, a, b \rangle
  +
  \langle \langle e, c, d \rangle, a, b \rangle
  \\[2pt]
  \;\;
  {} -
  \langle e, d, \langle c, b, a \rangle \rangle
  +
  \langle e, d, \langle c, a, b \rangle \rangle
  +
  \langle e, \langle c, b, a \rangle, d \rangle
  -
  \langle e, \langle c, a, b \rangle, d \rangle
  \\[2pt]
  \;\;
  {} -
  \langle e, \langle d, b, a \rangle, c \rangle
  +
  \langle e, \langle d, a, b \rangle, c \rangle
  +
  \langle e, c, \langle d, b, a \rangle \rangle
  -
  \langle e, c, \langle d, a, b \rangle \rangle
  \\[2pt]
  \;\;
  {} -
  \langle \langle e, b, a \rangle, d, c \rangle
  +
  \langle \langle e, a, b \rangle, d, c \rangle
  +
  \langle \langle e, b, a \rangle, c, d \rangle
  -
  \langle \langle e, a, b \rangle, c, d \rangle
  \equiv 0.
  \end{array}
  \right.
  \end{equation}
Part 2 produces 12 identities; after eliminating $[-,-,-]_2$ and $[-,-,-]_3$ we obtain:
  \begin{equation} \label{lts2}
  \left\{
  \begin{array}{l}
  \langle  a,  \langle  b, c, d  \rangle , e  \rangle
  +
  \langle  a,  \langle  c, b, d  \rangle , e  \rangle
  \equiv
  0,
  \\[2pt]
  \langle  a,  \langle  b, c, d  \rangle , e  \rangle
  +
  \langle  a,  \langle  c, d, b  \rangle , e  \rangle
  +
  \langle  a,  \langle  d, b, c  \rangle , e  \rangle
  \equiv
  0,
  \\[2pt]
  \langle  a, b,  \langle  c, d, e  \rangle   \rangle
  +
  \langle  a, b,  \langle  d, c, e  \rangle   \rangle
  \equiv
  0,
  \\[2pt]
  \langle  a, b,  \langle  c, d, e  \rangle   \rangle
  +
  \langle  a, b,  \langle  d, e, c  \rangle   \rangle
  +
  \langle  a, b,  \langle  e, c, d  \rangle   \rangle
  \equiv
  0.
  \end{array}
  \right.
  \end{equation}
These identities show that the inner triple in a monomial of the second or third association types,
$\langle -, \langle -, -, - \rangle, - \rangle$ and $\langle -, -, \langle -, -, - \rangle \rangle$,
has properties analogous to identities of degree 3 in the definition of Lie triple system:
the ternary analogues of skew-symmetry and the Jacobi identity.

\subsection{BSO algorithm}

We saw in Example \ref{ltsexample} that we need only one operation,
  \[
  \langle a,b,c \rangle = \widehat{a}bc - b\widehat{a}c - c\widehat{a}b + cb\widehat{a}.
  \]
Every identity of degree at most 5 satisfied by this operation follows from the two identities in the next definition.
Furthermore, the seven identities \eqref{lts1} and \eqref{lts2} are equivalent to the next two identities.

\begin{definition}
A \textbf{Leibniz triple system} (or Lie triple disystem) is a vector space $T$ with a trilinear
operation $\langle -,-,-\rangle\colon T \times T \times T \to T$
satisfying these identities:
  \allowdisplaybreaks
  \begin{align*}
  \langle a, \langle b, c, d \rangle, e \rangle
  &\equiv
  \langle \langle a, b, c \rangle, d, e \rangle
  -
  \langle \langle a, c, b \rangle, d, e \rangle
  -
  \langle \langle a, d, b \rangle, c, e \rangle
  +
  \langle \langle a, d, c \rangle, b, e \rangle,
  \\
  \langle a, b, \langle c, d, e \rangle \rangle
  &\equiv
  \langle \langle a, b, c \rangle, d, e \rangle
  -
  \langle \langle a, b, d \rangle, c, e \rangle
  -
  \langle \langle a, b, e \rangle, c, d \rangle
  +
  \langle \langle a, b, e \rangle, d, c \rangle.
  \end{align*}
In the right sides of these identities, the signs and permutations of $b,c,d$
and $c,d,e$ correspond to the expansion of the Lie triple products $[[b,c],d]$
and $[[c,d],e]$.
\end{definition}

\begin{theorem}
Any subspace of a Leibniz algebra which is closed under the iterated Leibniz bracket is a Leibniz triple system.
\end{theorem}

\begin{proof}
This follows from
$\langle a, b, c \rangle = ( a \dashv b - b \vdash a ) \dashv c - c \vdash ( a \dashv b - b \vdash a )$.
\end{proof}

\begin{theorem}
Every identity satisfied by the iterated Leibniz bracket
$\langle \langle a, b \rangle, c \rangle$ in every Leibniz algebra is a consequence of the defining identities
for Leibniz triple systems.
\end{theorem}

\begin{proof}
This follows from the construction in \cite{BSO2} of universal Leibniz envelopes for Leibniz triple systems.
\end{proof}

The next result from \cite{BSO2} generalizes the classical result that the associator in a Jordan algebra
satisfies the defining identities for Lie triple systems.

\begin{theorem}
Let $L$ be a subspace of a Jordan dialgebra which is closed under the associator $(a,b,c)$.
Then $L$ is a Leibniz triple system with the trilinear operation defined to be the permuted associator
$\langle a, b, c \rangle = (a,c,b)$.
\end{theorem}


\section{Jordan triple disystems}

We consider the dialgebra analogue of the variety of Jordan triple systems.
This section is a summary of Bremner, Felipe and S\'anchez-Ortega \cite{BFSO}.
We assume that the base field $\mathbb{F}$ does not have characteristic 2, 3 or 5.

\begin{definition}
A \textbf{Jordan triple system} is a vector space $T$ with a trilinear operation $T \times T \times T \to T$ denoted $\{-,-,-\}$
satisfying these identities:
  \allowdisplaybreaks
  \begin{align*}
  &
  \{a,b,c\} - \{c,b,a\} \equiv 0,
  \\
  &
  \{a,b,\{c,d,e\}\} - \{\{a,b,c\},d,e\} + \{c,\{b,a,d\},e\} - \{c,d,\{a,b,e\}\} \equiv 0.
  \end{align*}
These identities are satisfied by the Jordan triple product in any associative algebra.
\end{definition}

\subsection{KP algorithm}

We first consider Part 1 of the algorithm.
In the identity of degree 3, we make $a$, $b$, $c$ in turn the central argument and obtain
  \[
  \{a,b,c\}_1 - \{c,b,a\}_3 \equiv 0,
  \quad
  \{a,b,c\}_2 - \{c,b,a\}_2 \equiv 0,
  \quad
  \{a,b,c\}_3 - \{c,b,a\}_1 \equiv 0.
  \]
The third operation is superfluous and the second is symmetric in its first and third arguments.
In the identity of degree 5, we make $a,b,c,d,e$ in turn the central argument.
Replacing $\{a,b,c\}_3$ by $\{c,b,a\}_1$ in these five identities gives
  \allowdisplaybreaks
  \begin{align*}
  &
  \{ a, b, \{ c, d, e \}_1 \}_1 -
  \{ \{ a, b, c \}_1, d, e \}_1 +
  \{ c, \{ b, a, d \}_2, e \}_2 -
  \{ \{ a, b, e \}_1, d, c \}_1
  \equiv 0,
  \\
  &
  \{ a, b, \{ c, d, e \}_1 \}_2 -
  \{ \{ a, b, c \}_2, d, e \}_1 +
  \{ c, \{ b, a, d \}_1, e \}_2 -
  \{ \{ a, b, e \}_2, d, c \}_1
  \equiv 0,
  \\
  &
  \{ \{ c, d, e \}_1, b, a \}_1 -
  \{ \{ c, b, a \}_1, d, e \}_1 +
  \{ c, \{ b, a, d \}_1, e \}_1 -
  \{ c, d, \{ a, b, e \}_1 \}_1
  \equiv 0,
  \\
  &
  \{ \{ c, d, e \}_2, b, a \}_1 -
  \{ \{ c, b, a \}_1, d, e \}_2 +
  \{ c, \{ d, a, b \}_1, e \}_2 -
  \{ c, d, \{ a, b, e \}_1 \}_2
  \equiv 0,
  \\
  &
  \{ \{ e, d, c \}_1, b, a \}_1 -
  \{ e, d, \{ c, b, a \}_1 \}_1 +
  \{ e, \{ d, a, b \}_1, c \}_1 -
  \{ \{ e, b, a \}_1, d, c \}_1
  \equiv 0.
  \end{align*}
Part 2 of the algorithm produces the following identities, in which we have replaced $\{a,b,c\}_3$ by $\{c,b,a\}_1$:
  \allowdisplaybreaks
\begin{align*}
&
\{ a, \{ b, c, d \}_1, e \}_1 \equiv
\{ a, \{ b, c, d \}_2, e \}_1 \equiv
\{ a, \{ d, c, b \}_1, e \}_1,
\\
&
\{ a, b, \{ c, d, e \}_1 \}_1 \equiv
\{ a, b, \{ c, d, e \}_2 \}_1 \equiv
\{ a, b, \{ e, d, c \}_1 \}_1,
\\
&
\{ \{ a, b, c \}_1, d, e \}_2 \equiv
\{ \{ a, b, c \}_2, d, e \}_2 \equiv
\{ \{ c, b, a \}_1, d, e \}_2,
\\
&
\{ a, b, \{ c, d, e \}_1 \}_2 \equiv
\{ a, b, \{ c, d, e \}_2 \}_2 \equiv
\{ a, b, \{ e, d, c \}_1 \}_2.
\end{align*}

\begin{definition}
A \textbf{Jordan triple disystem} is a vector space
with trilinear operations $\{-,-,-\}_1$ and $\{-,-,-\}_2$
satisfying these eight identities:
  \allowdisplaybreaks
\begin{align*}
&
\{a,b,c\}_2 \equiv \{c,b,a\}_2,
\qquad
\{ \{ a, b, c \}_1, d, e \}_2 \equiv
\{ \{ a, b, c \}_2, d, e \}_2,
\\
&
\{ a, \{ b, c, d \}_1, e \}_1 \equiv
\{ a, \{ b, c, d \}_2, e \}_1,
\qquad
\{ a, b, \{ c, d, e \}_1 \}_1 \equiv
\{ a, b, \{ c, d, e \}_2 \}_1,
\\
&
\{ \{ e, d, c \}_1, b, a \}_1
\equiv
\{ \{ e, b, a \}_1, d, c \}_1
-
\{ e, \{ d, a, b \}_1, c \}_1
+
\{ e, d, \{ c, b, a \}_1 \}_1,
\\
&
\{ \{ e, d, c \}_2, b, a \}_1
\equiv
\{ \{ e, b, a \}_1, d, c \}_2
-
\{ e, \{ d, a, b \}_1, c \}_2
+
\{ e, d, \{ c, b, a \}_1 \}_2,
\\
&
\{ a, b, \{ c, d, e \}_1 \}_1
\equiv
\{ \{ a, b, c \}_1, d, e \}_1
-
\{ c, \{ b, a, d \}_2, e \}_2
+
\{ \{ a, b, e \}_1, d, c \}_1,
\\
&
\{ a, b, \{ c, d, e \}_1 \}_2
\equiv
\{ \{ a, b, c \}_2, d, e \}_1
-
\{ c, \{ b, a, d \}_1, e \}_2
+
\{ \{ a, b, e \}_2, d, c \}_1.
\end{align*}
\end{definition}

\subsection{BSO algorithm}

We saw in Example \ref{jtsexample} that we need only two operations,
  \[
  (a,b,c)_1 = \widehat{a}bc + cb\widehat{a},
  \qquad
  (a,b,c)_2 = a\widehat{b}c + c\widehat{b}a.
  \]
We use computer algebra to find the identities of low degree for these operations.

\begin{lemma}
Operation $(-,-,-)_1$ satisfies no polynomial identity of degree 3.
\end{lemma}

\begin{proof}
An identity is a linear combination of the six permutations of $(a,b,c)_1$:
  \allowdisplaybreaks
  \begin{align*}
  &
  x_1 (a,b,c)_1 +
  x_2 (a,c,b)_1 +
  x_3 (b,a,c)_1 +
  x_4 (b,c,a)_1 +
  x_5 (c,a,b)_1 +
  x_6 (c,b,a)_1
  \equiv 0.
  \end{align*}
We expand each ternary monomial to obtain a linear combination of the 18 multilinear dialgebra monomials of degree 3
ordered as follows:
\[
\widehat{a}bc, \,
\widehat{a}cb, \,
\widehat{b}ac, \,
\widehat{b}ca, \,
\widehat{c}ab, \,
\widehat{c}ba, \,
a\widehat{b}c, \,
a\widehat{c}b, \,
b\widehat{a}c, \,
b\widehat{c}a, \,
c\widehat{a}b, \,
c\widehat{b}a, \,
ab\widehat{c}, \,
ac\widehat{b}, \,
ba\widehat{c}, \,
bc\widehat{a}, \,
ca\widehat{b}, \,
cb\widehat{a}.
\]
We construct an $18 \times 6$ matrix $E$ whose $(i,j)$ entry is
the coefficient of the $i$-th dialgebra monomial in the expansion of the $j$-th diproduct monomial:
\[ \tiny
E^t
=
\left[
\begin{array}{cccccccccccccccccc}
1 & . & . & . & . & . & . & . & . & . & . & . & . & . & . & . & . & 1 \\
. & 1 & . & . & . & . & . & . & . & . & . & . & . & . & . & 1 & . & . \\
. & . & 1 & . & . & . & . & . & . & . & . & . & . & . & . & . & 1 & . \\
. & . & . & 1 & . & . & . & . & . & . & . & . & . & 1 & . & . & . & . \\
. & . & . & . & 1 & . & . & . & . & . & . & . & . & . & 1 & . & . & . \\
. & . & . & . & . & 1 & . & . & . & . & . & . & 1 & . & . & . & . & .
\end{array}
\right]
\]
The coefficient vectors of the polynomial identities satisfied by $(-,-,-)_1$
are the vectors in the nullspace of $E$, which is zero.
\end{proof}

\begin{lemma} \label{operation2degree3}
Every polynomial identity of degree 3 satisfied by operation $(-,-,-)_2$
is a consequence of $(a,b,c)_2 \equiv (c,b,a)_2$.
\end{lemma}

\begin{proof}
Following the same method as in the previous Lemma gives
\[ \tiny
E^t
=
\left[
\begin{array}{cccccccccccccccccc}
. & . & . & . & . & . & 1 & . & . & . & . & 1 & . & . & . & . & . & . \\
. & . & . & . & . & . & . & 1 & . & 1 & . & . & . & . & . & . & . & . \\
. & . & . & . & . & . & . & . & 1 & . & 1 & . & . & . & . & . & . & . \\
. & . & . & . & . & . & . & 1 & . & 1 & . & . & . & . & . & . & . & . \\
. & . & . & . & . & . & . & . & 1 & . & 1 & . & . & . & . & . & . & . \\
. & . & . & . & . & . & 1 & . & . & . & . & 1 & . & . & . & . & . & .
\end{array}
\right]
\]
The canonical basis of the nullspace consists of three vectors representing
the three permutations of the stated identity.
\end{proof}

These computations were extended to degree 5 to produce
a list of identities satisfied by $(-,-,-)_1$ and $(-,-,-)_2$ separately and together,
such that every identity of degree at most 5 satisfied by these operations
follows from the identities in the list.
It can then be verified that these identities are equivalent to the defining identities
for Jordan triple disystems.
In this way we obtain a large class of examples of special Jordan triple disystems.

\begin{theorem}
If $D$ is a subspace of an associative dialgebra which is closed under the Jordan diproducts $(-,-,-)_1$ and $(-,-,-)_2$,
then $D$ is a Jordan triple disystem with respect to these operations.
\end{theorem}

\subsection{Jordan dialgebras and Jordan triple disystems}

A Jordan algebra with product $a \circ b$ becomes a Jordan triple system
by means of the trilinear operation
  \[
  \langle a,b,c \rangle
  =
  (a \circ b) \circ c - (a \circ c) \circ b + a \circ (b \circ c).
  \]
Similarly, a Jordan dialgebra with operation $ab$ becomes a Jordan triple disystem
by means of two trilinear operations; the first is obtained by replacing $a \circ b$ by $ab$:
  \[
  \langle a,b,c \rangle_1 = (ab)c - (ac)b + a(bc),
  \qquad
  \langle a,b,c \rangle_2 = (ba)c + (bc)a - b(ac).
  \]
In a special Jordan dialgebra, we have $ab = a \dashv b + b \vdash a$,
and these two operations reduce (up to a scalar multiple) to the first and second dialgebra operations
in Example \ref{jtsexample}, namely $2 ( \widehat{a}bc + cb\widehat{a} )$ and
$2 ( a\widehat{b}c + c\widehat{b}a )$.
This construction provides a larger class of examples of Jordan triple disystems.

\begin{theorem}
If $D$ is a subspace of a Jordan dialgebra
which is closed under the trilinear operations $\langle -,-,- \rangle_1$ and $\langle -,-,- \rangle_2$,
then $D$ is a Jordan triple disystem with respect to these operations.
\end{theorem}

\begin{proof}
This is a sketch of a computational proof of this result, starting with degree 3.
We must show that every polynomial identity of degree 3 satisfied by
$\langle a,b,c \rangle_1$ and $\langle a,b,c \rangle_2$
follows from the symmetry of $\langle a,b,c \rangle_2$ in its first and third arguments.
We construct an $18 \times 24$ matrix $E$ in which
columns 1--12 correspond to the 12 multilinear monomials of degree 3
in the free nonassociative algebra,
  \[
  (ab)c, \, (ac)b, \, (ba)c, \, (bc)a, \, (ca)b, \, (cb)a, \,
  a(bc), \, a(cb), \, b(ac), \, b(ca), \, c(ab), \, c(ba),
  \]
and columns 13--24 correspond to the 12 trilinear monomials of degree 3 in the trilinear operations
$\langle \cdots \rangle_1$ and $\langle \cdots \rangle_2$,
  \[
  \begin{array}{llllll}
  \langle a,b,c \rangle_1, &\;
  \langle a,c,b \rangle_1, &\;
  \langle b,a,c \rangle_1, &\;
  \langle b,c,a \rangle_1, &\;
  \langle c,a,b \rangle_1, &\;
  \langle c,b,a \rangle_1,
  \\
  \langle a,b,c \rangle_2, &\;
  \langle a,c,b \rangle_2, &\;
  \langle b,a,c \rangle_2, &\;
  \langle b,c,a \rangle_2, &\;
  \langle c,a,b \rangle_2, &\;
  \langle c,b,a \rangle_2.
  \end{array}
  \]
The matrix $E$ has the following block structure,
  \[
  E
  =
  \left[
  \begin{array}{cc}
  R & O \\
  X & I
  \end{array}
  \right],
  \]
and its entries are determined as follows:
  \begin{itemize}
  \item the upper left $6 \times 12$ block $R$ contains the coefficient vectors of
  the permutations of the right commutative identity;
  \item the lower left $12 \times 12$ block $X$ contains the coefficient vectors of
  the expansions of the operations $\langle -,-,- \rangle_1$ and $\langle -,-,- \rangle_2$;
  \item the upper right $6 \times 12$ block $O$ contains the zero matrix;
  \item the lower right $12 \times 12$ block $I$ contains the identity matrix.
  \end{itemize}
This matrix is displayed in Table \ref{JTSmatrixE} using $\cdot, +, -$ for $0, 1, -1$.
The row canonical form is displayed in Table \ref{JTSmatrixErcf} using $\ast$ for $\tfrac12$;
the rank is 15.
The dividing line between the upper and lower parts of the row canonical form lies immediately above row 13:
the uppermost row whose leading 1 is in the right part of the matrix.
The rows below this line represent the dependence relations among the expansions of the trilinear monomials
which hold as a result of the right commutative identities.
The rows of the lower right $3 \times 12$ block represent the
permutations of $\langle a,b,c \rangle_2 - \langle c,b,a \rangle_2 \equiv 0$.
\end{proof}

  \begin{table}
  \begin{center} \small
  \[
  \left[
  \begin{array}{cccccccccccc|cccccccccccc}
  . &\!\!\! . &\!\!\! . &\!\!\! . &\!\!\! . &\!\!\! . &\!\!\! + &\!\!\! - &\!\!\! . &\!\!\! . &\!\!\! . &\!\!\! . &
  . &\!\!\! . &\!\!\! . &\!\!\! . &\!\!\! . &\!\!\! . &\!\!\! . &\!\!\! . &\!\!\! . &\!\!\! . &\!\!\! . &\!\!\! . \\
  . &\!\!\! . &\!\!\! . &\!\!\! . &\!\!\! . &\!\!\! . &\!\!\! - &\!\!\! + &\!\!\! . &\!\!\! . &\!\!\! . &\!\!\! . &
  . &\!\!\! . &\!\!\! . &\!\!\! . &\!\!\! . &\!\!\! . &\!\!\! . &\!\!\! . &\!\!\! . &\!\!\! . &\!\!\! . &\!\!\! . \\
  . &\!\!\! . &\!\!\! . &\!\!\! . &\!\!\! . &\!\!\! . &\!\!\! . &\!\!\! . &\!\!\! + &\!\!\! - &\!\!\! . &\!\!\! . &
  . &\!\!\! . &\!\!\! . &\!\!\! . &\!\!\! . &\!\!\! . &\!\!\! . &\!\!\! . &\!\!\! . &\!\!\! . &\!\!\! . &\!\!\! . \\
  . &\!\!\! . &\!\!\! . &\!\!\! . &\!\!\! . &\!\!\! . &\!\!\! . &\!\!\! . &\!\!\! - &\!\!\! + &\!\!\! . &\!\!\! . &
  . &\!\!\! . &\!\!\! . &\!\!\! . &\!\!\! . &\!\!\! . &\!\!\! . &\!\!\! . &\!\!\! . &\!\!\! . &\!\!\! . &\!\!\! . \\
  . &\!\!\! . &\!\!\! . &\!\!\! . &\!\!\! . &\!\!\! . &\!\!\! . &\!\!\! . &\!\!\! . &\!\!\! . &\!\!\! + &\!\!\! - &
  . &\!\!\! . &\!\!\! . &\!\!\! . &\!\!\! . &\!\!\! . &\!\!\! . &\!\!\! . &\!\!\! . &\!\!\! . &\!\!\! . &\!\!\! . \\
  . &\!\!\! . &\!\!\! . &\!\!\! . &\!\!\! . &\!\!\! . &\!\!\! . &\!\!\! . &\!\!\! . &\!\!\! . &\!\!\! - &\!\!\! + &
  . &\!\!\! . &\!\!\! . &\!\!\! . &\!\!\! . &\!\!\! . &\!\!\! . &\!\!\! . &\!\!\! . &\!\!\! . &\!\!\! . &\!\!\! . \\
  \midrule
  + &\!\!\! - &\!\!\! . &\!\!\! . &\!\!\! . &\!\!\! . &\!\!\! + &\!\!\! . &\!\!\! . &\!\!\! . &\!\!\! . &\!\!\! . &
  + &\!\!\! . &\!\!\! . &\!\!\! . &\!\!\! . &\!\!\! . &\!\!\! . &\!\!\! . &\!\!\! . &\!\!\! . &\!\!\! . &\!\!\! . \\
  - &\!\!\! + &\!\!\! . &\!\!\! . &\!\!\! . &\!\!\! . &\!\!\! . &\!\!\! + &\!\!\! . &\!\!\! . &\!\!\! . &\!\!\! . &
  . &\!\!\! + &\!\!\! . &\!\!\! . &\!\!\! . &\!\!\! . &\!\!\! . &\!\!\! . &\!\!\! . &\!\!\! . &\!\!\! . &\!\!\! . \\
  . &\!\!\! . &\!\!\! + &\!\!\! - &\!\!\! . &\!\!\! . &\!\!\! . &\!\!\! . &\!\!\! + &\!\!\! . &\!\!\! . &\!\!\! . &
  . &\!\!\! . &\!\!\! + &\!\!\! . &\!\!\! . &\!\!\! . &\!\!\! . &\!\!\! . &\!\!\! . &\!\!\! . &\!\!\! . &\!\!\! . \\
  . &\!\!\! . &\!\!\! - &\!\!\! + &\!\!\! . &\!\!\! . &\!\!\! . &\!\!\! . &\!\!\! . &\!\!\! + &\!\!\! . &\!\!\! . &
  . &\!\!\! . &\!\!\! . &\!\!\! + &\!\!\! . &\!\!\! . &\!\!\! . &\!\!\! . &\!\!\! . &\!\!\! . &\!\!\! . &\!\!\! . \\
  . &\!\!\! . &\!\!\! . &\!\!\! . &\!\!\! + &\!\!\! - &\!\!\! . &\!\!\! . &\!\!\! . &\!\!\! . &\!\!\! + &\!\!\! . &
  . &\!\!\! . &\!\!\! . &\!\!\! . &\!\!\! + &\!\!\! . &\!\!\! . &\!\!\! . &\!\!\! . &\!\!\! . &\!\!\! . &\!\!\! . \\
  . &\!\!\! . &\!\!\! . &\!\!\! . &\!\!\! - &\!\!\! + &\!\!\! . &\!\!\! . &\!\!\! . &\!\!\! . &\!\!\! . &\!\!\! + &
  . &\!\!\! . &\!\!\! . &\!\!\! . &\!\!\! . &\!\!\! + &\!\!\! . &\!\!\! . &\!\!\! . &\!\!\! . &\!\!\! . &\!\!\! . \\
  . &\!\!\! . &\!\!\! + &\!\!\! + &\!\!\! . &\!\!\! . &\!\!\! . &\!\!\! . &\!\!\! - &\!\!\! . &\!\!\! . &\!\!\! . &
  . &\!\!\! . &\!\!\! . &\!\!\! . &\!\!\! . &\!\!\! . &\!\!\! + &\!\!\! . &\!\!\! . &\!\!\! . &\!\!\! . &\!\!\! . \\
  . &\!\!\! . &\!\!\! . &\!\!\! . &\!\!\! + &\!\!\! + &\!\!\! . &\!\!\! . &\!\!\! . &\!\!\! . &\!\!\! - &\!\!\! . &
  . &\!\!\! . &\!\!\! . &\!\!\! . &\!\!\! . &\!\!\! . &\!\!\! . &\!\!\! + &\!\!\! . &\!\!\! . &\!\!\! . &\!\!\! . \\
  + &\!\!\! + &\!\!\! . &\!\!\! . &\!\!\! . &\!\!\! . &\!\!\! - &\!\!\! . &\!\!\! . &\!\!\! . &\!\!\! . &\!\!\! . &
  . &\!\!\! . &\!\!\! . &\!\!\! . &\!\!\! . &\!\!\! . &\!\!\! . &\!\!\! . &\!\!\! + &\!\!\! . &\!\!\! . &\!\!\! . \\
  . &\!\!\! . &\!\!\! . &\!\!\! . &\!\!\! + &\!\!\! + &\!\!\! . &\!\!\! . &\!\!\! . &\!\!\! . &\!\!\! . &\!\!\! - &
  . &\!\!\! . &\!\!\! . &\!\!\! . &\!\!\! . &\!\!\! . &\!\!\! . &\!\!\! . &\!\!\! . &\!\!\! + &\!\!\! . &\!\!\! . \\
  + &\!\!\! + &\!\!\! . &\!\!\! . &\!\!\! . &\!\!\! . &\!\!\! . &\!\!\! - &\!\!\! . &\!\!\! . &\!\!\! . &\!\!\! . &
  . &\!\!\! . &\!\!\! . &\!\!\! . &\!\!\! . &\!\!\! . &\!\!\! . &\!\!\! . &\!\!\! . &\!\!\! . &\!\!\! + &\!\!\! . \\
  . &\!\!\! . &\!\!\! + &\!\!\! + &\!\!\! . &\!\!\! . &\!\!\! . &\!\!\! . &\!\!\! . &\!\!\! - &\!\!\! . &\!\!\! . &
  . &\!\!\! . &\!\!\! . &\!\!\! . &\!\!\! . &\!\!\! . &\!\!\! . &\!\!\! . &\!\!\! . &\!\!\! . &\!\!\! . &\!\!\! +
  \end{array}
  \right]
  \]
  \end{center}
  \caption{The matrix $E$}
  \label{JTSmatrixE}
  \end{table}

  \begin{table}
  \begin{center} \small
  \[
  \left[
  \begin{array}{cccccccccccc|cccccccccccc}
  + &\!\! . &\!\! . &\!\! . &\!\! . &\!\! . &\!\! . &\!\! . &\!\! . &\!\! . &\!\! . &\!\! . &
  \ast &\!\! . &\!\! . &\!\! . &\!\! . &\!\! . &\!\! . &\!\! . &\!\! . &\!\! . &\!\! \ast &\!\! . \\
  . &\!\! + &\!\! . &\!\! . &\!\! . &\!\! . &\!\! . &\!\! . &\!\! . &\!\! . &\!\! . &\!\! . &
   . &\!\! \ast &\!\! . &\!\! . &\!\! . &\!\! . &\!\! . &\!\! . &\!\! . &\!\! . &\!\! \ast &\!\! . \\
  . &\!\! . &\!\! + &\!\! . &\!\! . &\!\! . &\!\! . &\!\! . &\!\! . &\!\! . &\!\! . &\!\! . &
   . &\!\! . &\!\! \ast &\!\! . &\!\! . &\!\! . &\!\! . &\!\! . &\!\! . &\!\! . &\!\! . &\!\! \ast \\
  . &\!\! . &\!\! . &\!\! + &\!\! . &\!\! . &\!\! . &\!\! . &\!\! . &\!\! . &\!\! . &\!\! . &
   . &\!\! . &\!\! . &\!\! \ast &\!\! . &\!\! . &\!\! . &\!\! . &\!\! . &\!\! . &\!\! . &\!\! \ast \\
  . &\!\! . &\!\! . &\!\! . &\!\! + &\!\! . &\!\! . &\!\! . &\!\! . &\!\! . &\!\! . &\!\! . &
   . &\!\! . &\!\! . &\!\! . &\!\! \ast &\!\! . &\!\! . &\!\! . &\!\! . &\!\! \ast &\!\! . &\!\! . \\
  . &\!\! . &\!\! . &\!\! . &\!\! . &\!\! + &\!\! . &\!\! . &\!\! . &\!\! . &\!\! . &\!\! . &
   . &\!\! . &\!\! . &\!\! . &\!\! . &\!\! \ast &\!\! . &\!\! . &\!\! . &\!\! \ast &\!\! . &\!\! . \\
  . &\!\! . &\!\! . &\!\! . &\!\! . &\!\! . &\!\! + &\!\! . &\!\! . &\!\! . &\!\! . &\!\! . &
   \ast &\!\! \ast &\!\! . &\!\! . &\!\! . &\!\! . &\!\! . &\!\! . &\!\! . &\!\! . &\!\! . &\!\! . \\
  . &\!\! . &\!\! . &\!\! . &\!\! . &\!\! . &\!\! . &\!\! + &\!\! . &\!\! . &\!\! . &\!\! . &
   \ast &\!\! \ast &\!\! . &\!\! . &\!\! . &\!\! . &\!\! . &\!\! . &\!\! . &\!\! . &\!\! . &\!\! . \\
  . &\!\! . &\!\! . &\!\! . &\!\! . &\!\! . &\!\! . &\!\! . &\!\! + &\!\! . &\!\! . &\!\! . &
   . &\!\! . &\!\! \ast &\!\! \ast &\!\! . &\!\! . &\!\! . &\!\! . &\!\! . &\!\! . &\!\! . &\!\! . \\
  . &\!\! . &\!\! . &\!\! . &\!\! . &\!\! . &\!\! . &\!\! . &\!\! . &\!\! + &\!\! . &\!\! . &
   . &\!\! . &\!\! \ast &\!\! \ast &\!\! . &\!\! . &\!\! . &\!\! . &\!\! . &\!\! . &\!\! . &\!\! . \\
  . &\!\! . &\!\! . &\!\! . &\!\! . &\!\! . &\!\! . &\!\! . &\!\! . &\!\! . &\!\! + &\!\! . &
   . &\!\! . &\!\! . &\!\! . &\!\! \ast &\!\! \ast &\!\! . &\!\! . &\!\! . &\!\! . &\!\! . &\!\! . \\
  . &\!\! . &\!\! . &\!\! . &\!\! . &\!\! . &\!\! . &\!\! . &\!\! . &\!\! . &\!\! . &\!\! + &
   . &\!\! . &\!\! . &\!\! . &\!\! \ast &\!\! \ast &\!\! . &\!\! . &\!\! . &\!\! . &\!\! . &\!\! . \\
  \midrule
  . &\!\! . &\!\! . &\!\! . &\!\! . &\!\! . &\!\! . &\!\! . &\!\! . &\!\! . &\!\! . &\!\! . &
   . &\!\! . &\!\! . &\!\! . &\!\! . &\!\! . &\!\! + &\!\! . &\!\! . &\!\! . &\!\! . &\!\! - \\
  . &\!\! . &\!\! . &\!\! . &\!\! . &\!\! . &\!\! . &\!\! . &\!\! . &\!\! . &\!\! . &\!\! . &
   . &\!\! . &\!\! . &\!\! . &\!\! . &\!\! . &\!\! . &\!\! + &\!\! . &\!\! - &\!\! . &\!\! . \\
  . &\!\! . &\!\! . &\!\! . &\!\! . &\!\! . &\!\! . &\!\! . &\!\! . &\!\! . &\!\! . &\!\! . &
   . &\!\! . &\!\! . &\!\! . &\!\! . &\!\! . &\!\! . &\!\! . &\!\! + &\!\! . &\!\! - &\!\! .
  \end{array}
  \right]
  \]
  \end{center}
  \caption{The row canonical form of $E$}
  \label{JTSmatrixErcf}
  \end{table}

We can extend these computations to degree 5; the matrix $E$ has the same block structure but is much larger.
In degree 5, there are $5!$ permutations of the variables, and
14 association types for a nonassociative binary operation,
  \[
  \begin{array}{lllll}
  (((ab)c)d)e, &\, ((a(bc))d)e, &\, ((ab)(cd))e, &\, (a((bc)d))e, &\, (a(b(cd)))e, \\
  ((ab)c)(de), &\, (a(bc))(de), &\, (ab)((cd)e), &\, (ab)(c(de)), &\, a(((bc)d)e), \\
  a((b(cd))e), &\, a((bc)(de)), &\, a(b((cd)e)), &\, a(b(c(de))),
  \end{array}
  \]
giving 1680 monomials labeling the columns in the left part.
There are 10 association types in degree 5 for two trilinear operations,
assuming that the second operation is symmetric in its first and third arguments:
\[
\begin{array}{llll}
\langle \langle a,b,c\rangle_1,d,e\rangle_1, &\quad
\langle a,\langle b,c,d\rangle_1,e\rangle_1, &\quad
\langle a,b,\langle c,d,e\rangle_1\rangle_1, &\quad
\langle \langle a,b,c\rangle_2,d,e\rangle_2, \\
\langle a,\langle b,c,d\rangle_2,e\rangle_2, &\quad
\langle \langle a,b,c\rangle_2,d,e\rangle_1, &\quad
\langle a,\langle b,c,d\rangle_2,e\rangle_1, &\quad
\langle a,b,\langle c,d,e\rangle_2\rangle_1, \\
\langle \langle a,b,c\rangle_1,d,e\rangle_2, &\quad
\langle a,\langle b,c,d\rangle_1,e\rangle_2.
\end{array}
\]
Using the symmetry of $\langle -,-,- \rangle_2$ we obtain the number of multilinear monomials in each type,
giving $120 + 120 + 120 + 60 + 60 + 60 + 120 + 60 + 60 + 30 = 810$
monomials labeling the columns in the right part.

We next generate all the consequences in degree 5 of the defining identities for Jordan dialgebras.
A multilinear identity $I(a_1,\dots,a_n)$ of degree $n$ produces $n{+}2$ identities of degree $n{+}1$;
we have $n$ substitutions and two multiplications:
  \allowdisplaybreaks
  \begin{align*}
  I( a_1 a_{n+1}, \dots, a_n ),
  \;
  \dots,
  \;
  I( a_1, \dots, a_n a_{n+1} ),
  \;
  I( a_1, \dots, a_n ) a_{n+1},
  \;
  a_{n+1} I( a_1, \dots, a_n ).
  \end{align*}
The right commutative identity of degree 3 produces 5 identities of degree 4, and each of these
produces 6 identities of degree 5, for a total of 30.
The linearized versions of the right Osborn and right Jordan identities of degree 4 each produce
6 identities of degree 5, for a total of 12.
Altogether we have 42 identities of degree 5, and each allows $5!$ permutations of the variables,
for a total of 5040.
The upper left block $R$ of the matrix $E$ has size $5040 \times 1680$.

The lower left block $X$ has size $810 \times 1680$ and contains the coefficients of the expansions
of the ternary monomials.
The upper right block $O$ is the $5040 \times 810$ zero matrix,
and the lower right block $I$ is the $810 \times 810$ identity matrix:
  \[
  E =
  \left[
  \begin{array}{c|c}
  \begin{tabular}{l}
  consequences in degree 5 of \\
  the Jordan dialgebra identities
  \end{tabular}
  &
  \begin{tabular}{l}
  zero matrix
  \end{tabular}
  \\
  \midrule
  \begin{tabular}{l}
  expansions of the monomials \\
  in degree 5 for $\langle \cdots \rangle_1$ and $\langle \cdots \rangle_2$
  \end{tabular}
  &
  \begin{tabular}{l}
  identity matrix
  \end{tabular}
  \end{array}
  \right]
  \]
We compute the row canonical form and find that the rank is 2215.
We ignore the first 1655 rows since their leading 1s are in the left part;
we retain only the 560 rows which have their leading 1s in the right part.
We sort these rows by increasing number of nonzero components.
These rows represent the identities in degree 5 satisfied by the Jordan triple diproducts
in a Jordan dialgebra.

We construct another matrix $M$ with an upper block of size $810 \times 810$ and
a lower block of size $120 \times 810$.
For each of the 560 identities satisfied by the operations
$\langle -,-,- \rangle_1$ and $\langle -,-,- \rangle_2$,
we apply all $5!$ permutations of the variables, store the permuted identities in the lower block,
and compute the row canonical form.
We record the index numbers of the identities which increase the rank:
  \begin{center}
  \begin{tabular}{lrrrrrrrr}
  identity & 1 & 121 & 241 & 301 & 331 & 342 & 451 & 454 \\
  rank & 120 & 240 & 360 & 390 & 450 & 470 & 530 & 560
  \end{tabular}
  \end{center}
We then verify directly that these eight identities generate the same $S_5$-module as
the defining identities for Jordan triple disystems obtained from the KP algorithm.


\section{The cyclic commutator}

In this section, we present some new results about the trilinear operation which we call the \textbf{cyclic commutator},
  \[
  (a,b,c) = abc - bca.
  \]
This operation provides a ``noncommutative'' version of Lie triple systems different from Leibniz triple systems.

\subsection{Polynomial identities}

The next result appears in Bremner and Peresi \cite{BremnerPeresi2} in a slightly different form.

\begin{lemma}
Every multilinear polynomial identity of degree 3 satisfied by the cyclic commutator follows from
the ternary Jacobi identity,
  \[
  (a,b,c) + (b,c,a) + (c,a,b) \equiv 0.
  \]
Every multilinear polynomial identity of degree 5 satisfied by the cyclic commutator follows from
the ternary Jacobi identity and the (right) ternary derivation identity,
  \[
  ((a,b,c),d,e) \equiv ((a,d,e),b,c) + (a,(b,d,e),c) + (a,b,(c,d,e)).
  \]
\end{lemma}

We now extend these computations to degree 7.
For a general trilinear operation, the number of association types in (odd) degree $n$
equals the number of ternary trees with $n$ leaf nodes; see sequence A001764 in Sloane \cite{Sloane}
and Example 5 on page 360 of Graham et al.~\cite{GKP}.
There is a simple formula for this number:
  \[
  t(k) = \frac{1}{2k+1} \binom{3k}{k}
  \qquad
  (n = 2k+1).
  \]
The first few values are as follows:
  \[
  \begin{array}{rrrrrrrrrrrr}
  k & 1 & 2 & 3 & 4 & 5 & 6 & 7 & 8 & 9 & 10 \\
  n & 3 & 5 & 7 & 9 & 11 & 13 & 15 & 17 & 19 & 21 \\
  t(k) & 1 & 3 & 12 & 55 & 273 & 1428 & 7752 & 43263 & 246675 & 1430715
  \end{array}
  \]
We order the 12 ternary association types in degree 7 as follows:
  \[
  \begin{array}{lll}
  (-,-,(-,-,(-,-,-))), &
  (-,-,(-,(-,-,-),-)), &
  (-,-,((-,-,-),-,-)),
  \\
  (-,(-,-,-),(-,-,-)), &
  (-,(-,-,(-,-,-)),-), &
  (-,(-,(-,-,-),-),-),
  \\
  (-,((-,-,-),-,-),-), &
  ((-,-,-),-,(-,-,-)), &
  ((-,-,-),(-,-,-),-),
  \\
  ((-,-,(-,-,-)),-,-), &
  ((-,(-,-,-),-),-,-), &
  (((-,-,-),-,-),-,-).
  \end{array}
  \]
Using the ternary Jacobi identity, we can eliminate types 3, 7, 9, 10, 11, 12
by means of the following equations and retain only types 1, 2, 4, 5, 6, 8:
  \begin{equation}
  \label{reducetype}
  \left\{
  \begin{array}{l}
  (a,b,((c,d,e),f,g)) = - (a,b,(f,g,(c,d,e))) - (a,b,(g,(c,d,e),f)),
  \\
  (a,((b,c,d),e,f),g) = - (a,(e,f,(b,c,d)),g) - (a,(f,(b,c,d),e),g),
  \\
  ((a,b,c),(d,e,f),g) = - ((d,e,f),g,(a,b,c)) - (g,(a,b,c),(d,e,f)),
  \\
  ((a,b,(c,d,e)),f,g) = - (f,g,(a,b,(c,d,e))) - (g,(a,b,(c,d,e)),f),
  \\
  ((a,(b,c,d),e),f,g) = - (f,g,(a,(b,c,d),e)) - (g,(a,(b,c,d),e),f),
  \\
  (((a,b,c),d,e),f,g) = (f,g,(d,e,(a,b,c))) + (f,g,(e,(a,b,c),d))
  \\
  \hfill
  + (g,(d,e,(a,b,c)),f) + (g,(e,(a,b,c),d),f).
  \end{array}
  \right.
  \end{equation}
Using the ternary Jacobi identity again, we can further reduce multilinear monomials in the remaining 6 types
by means of the following equations:
  \begin{equation}
  \label{reduceperm}
  \left\{
  \begin{array}{l}
  (a,b,(c,d,(g,e,f))) = - (a,b,(c,d,(e,f,g))) - (a,b,(c,d,(f,g,e))),
  \\
  (a,b,(c,(f,d,e),g)) = - (a,b,(c,(d,e,f),g)) - (a,b,(c,(e,f,d),g)),
  \\
  (a,(d,b,c),(g,e,f)) = - (a,(b,c,d),(g,e,f)) - (a,(c,d,b),(g,e,f)),
  \\
  \hfill
  = - (a,(d,b,c),(e,f,g)) - (a,(d,b,c),(f,g,e)),
  \\
  (a,(b,c,(f,d,e)),g) = - (a,(b,c,(d,e,f)),g) - (a,(b,c,(e,f,d)),g),
  \\
  (a,(b,(e,c,d),f),g) = - (a,(b,(c,d,e),f),g) - (a,(b,(d,e,c),f),g),
  \\
  ((c,a,b),d,(g,e,f)) = - ((a,b,c),d,(g,e,f)) - ((b,c,a),d,(g,e,f)),
  \\
  \hfill
  = - ((c,a,b),d,(e,f,g)) - ((c,a,b),d,(f,g,e)).
  \end{array}
  \right.
  \end{equation}
The basic principle is that when all three arguments have degree 1,
the first argument should not lexicographically follow both the second and third arguments.
It follows that the total number of multilinear monomials in degree 7 equals
  \[
  \Big( \tfrac23 + \tfrac23 + (\tfrac23)^2 +  \tfrac23 + \tfrac23 + (\tfrac23)^2 \Big) \cdot 7!
  =
  17920.
  \]
In order to prove that these multilinear monomials are linearly independent, we first write the ternary Jacobi identity as follows:
  \[
  I(a,b,c) = (a,b,c) + (b,c,a) + (c,a,b).
  \]
We consider the following consequences of $I(a,b,c)$ in degree 5:
  \begin{equation}
  \label{lift3}
  \left\{
  \begin{array}{lll}
  I((a,d,e),b,c), &\qquad
  I(a,(b,d,e),c), &\qquad
  I(a,b,(c,d,e)),
  \\
  (I(a,b,c),d,e), &\qquad
  (d,I(a,b,c),e), &\qquad
  (d,e,I(a,b,c)).
  \end{array}
  \right.
  \end{equation}
Every consequence of $I(a,b,c)$ in degree 5 is a linear combination of permutations
of these 6 identities.
We write $J(a,b,c,d,e)$ for one of these identities.
We consider the following 8 consequences of $J(a,b,c,d,e)$ in degree 7:
  \begin{equation}
  \label{lift5}
  \left\{
  \begin{array}{lll}
  J((a,f,g),b,c,d,e),
  &\quad
  J(a,(b,f,g),c,d,e),
  &\quad
  J(a,b,(c,f,g),d,e),
  \\
  J(a,b,c,(d,f,g),e),
  &\quad
  J(a,b,c,d,(e,f,g)),
  &\quad
  (J(a,b,c,d,e),f,g),
  \\
  (f,J(a,b,c,d,e),g),
  &\quad
  (f,g,J(a,b,c,d,e)).
  \end{array}
  \right.
  \end{equation}
Every consequence of $I(a,b,c)$ in degree 7 is a linear combination of permutations
of the resulting 48 identities.
We now reduce each of these identities in degree 7 using equations
\eqref{reducetype} and \eqref{reduceperm}, and verify that in every case the
result collapses to 0.
This proves that the multilinear monomials are linearly independent,
and hence form a basis for the multilinear subspace of degree 7 in the free
ternary algebra in the variety defined by the ternary Jacobi identity.

We now write the ternary derivation identity in the form
  \[
  J(a,b,c,d,e) = ((a,b,c),d,e) - ((a,d,e),b,c) - (a,(b,d,e),c) - (a,b,(c,d,e)),
  \]
and consider its consequences in degree 7 using \eqref{lift5}.
Every consequence in degree 7 is a linear combination of permutations
of these 8 identities;
we reduce each of them using equations
\eqref{reducetype} and \eqref{reduceperm}.
We create a matrix of size $22960 \times 17920$ with an upper block of size $17920 \times 17920$ and
a lower block of size $5040 \times 17920$.
In order to control memory allocation, we use modular arithmetic with $p = 101$.
(Since the group algebra $\mathbb{F} S_n$ is semisimple for $p > n$, the structure constants
from characteristic 0 are well-defined for any $p > n$.  It follows that we can do these
computations using modular arithmetic with any $p > n$ and then use rational reconstruction to recover the
correct results for characteristic 0.)
For each of the 8 consequences of the ternary derivation identity,
we apply all 5040 permutations of the variables, store the coefficient vectors
of the resulting identities in the lower block, and compute the row canonical form.
At the end of this calculation, the matrix has rank 13372;
the row space of this matrix consists of the coefficient vectors
of all polynomial identities in degree 7 for the ternary commutator
which are consequences of the ternary derivation identity.

We construct another matrix of size $5040 \times 17920$;
in each column we put the coefficient vector of the expansion of the corresponding
ternary monomial into the free associative algebra using the ternary commutator.
The expansions for association types 1, 2, 4, 5, 6, 8 with the identity permutation are as follows:
  \allowdisplaybreaks
  \begin{align*}
  (a,b,(c,d,(e,f,g))) &=
     abcdefg
  -  bcdefga
  -  abdefgc
  +  bdefgca
  \\
  &\quad
  -  abcdfge
  +  bcdfgea
  +  abdfgec
  -  bdfgeca,
  \\
  (a,b,(c,(d,e,f),g)) &=
     abcdefg
  -  bcdefga
  -  abdefgc
  +  bdefgca
  \\
  &\quad
  -  abcefdg
  +  bcefdga
  +  abefdgc
  -  befdgca,
  \\
  (a,(b,c,d),(e,f,g)) &=
     abcdefg
  -  bcdefga
  -  abcdfge
  +  bcdfgea
  \\
  &\quad
  -  acdbefg
  +  cdbefga
  +  acdbfge
  -  cdbfgea,
  \\
  (a,(b,c,(d,e,f)),g) &=
     abcdefg
  -  bcdefga
  -  acdefbg
  +  cdefbga
  \\
  &\quad
  -  abcefdg
  +  bcefdga
  +  acefdbg
  -  cefdbga,
  \\
  (a,(b,(c,d,e),f),g) &=
     abcdefg
  -  bcdefga
  -  acdefbg
  +  cdefbga
  \\
  &\quad
  -  abdecfg
  +  bdecfga
  +  adecfbg
  -  decfbga,
  \\
  ((a,b,c),d,(e,f,g)) &=
     abcdefg
  -  defgabc
  -  abcdfge
  +  dfgeabc
  \\
  &\quad
  -  bcadefg
  +  defgbca
  +  bcadfge
  -  dfgebca.
  \end{align*}
Still using arithmetic modulo $p = 101$, we compute the row canonical form of this
matrix and extract the canonical basis of the nullspace.
The rank is 4128 and hence the dimension of the nullspace is 13792.
Comparing this result with that of the previous paragraph, we see that there is a
quotient space of dimension $13792 - 13372 = 420$
consisting of polynomial identities in degree 7 for the ternary commutator
which are not consequences of the identities of lower degree.
We sort these identities by increasing number of nonzero entries
in the coefficient vector.
Starting with the matrix of rank 13372 from the previous paragraph,
we process each identity in this sorted list by applying all 5040 permutations to the variables,
storing the results in the lower block, and reducing the matrix.
Only two identities increase the rank: an identity with 20 terms increases the rank to 13722,
and an identity with 45 terms increases the rank to 13792.
Further calculations show that the first identity is a consequence of the second.

The second identity is given in the following Theorem.
The fact that this identity is satisfied by the cyclic commutator can be verified
directly by expanding each term into the free associative algebra.
But to prove that this identity is not a consequence of the identities of lower degree
requires a computation such as that just described.

\begin{theorem} \label{degree7theorem}
Every multilinear polynomial identity of degree 7 satisfied by the cyclic commutator is
a consequence of the ternary Jacobi identity, the ternary derivation identity, and the
following identity with 45 terms and coefficients $\pm 1$:
  \allowdisplaybreaks
  \begin{align*}
  &
    ( a b ( c d ( e f g ) ) )
  - ( a b ( c f ( d e g ) ) )
  - ( a b ( c f ( e g d ) ) )
  - ( a b ( c e ( d g f ) ) )
  - ( a b ( c g ( f e d ) ) )
  \\
  &
  - ( a b ( d c ( f e g ) ) )
  + ( a b ( d f ( c g e ) ) )
  + ( a b ( d e ( c f g ) ) )
  + ( a b ( d e ( f g c ) ) )
  - ( a b ( f c ( d g e ) ) )
  \\
  &
  - ( a b ( f c ( e d g ) ) )
  - ( a b ( f d ( e g c ) ) )
  + ( a b ( f g ( d e c ) ) )
  + ( a b ( e c ( f g d ) ) )
  + ( a b ( e d ( c g f ) ) )
  \\
  &
  + ( a b ( e d ( f c g ) ) )
  - ( a b ( e g ( c f d ) ) )
  - ( a b ( e g ( d c f ) ) )
  + ( a b ( g f ( d c e ) ) )
  - ( a b ( g e ( c d f ) ) )
  \\
  &
  + ( a c ( f b ( d e g ) ) )
  - ( a c ( g b ( d f e ) ) )
  - ( a c ( g b ( f e d ) ) )
  - ( a d ( e b ( c f g ) ) )
  - ( a d ( g b ( f c e ) ) )
  \\
  &
  - ( a g ( d b ( c f e ) ) )
  - ( a g ( d b ( f e c ) ) )
  + ( a g ( f b ( c e d ) ) )
  - ( a ( b c d ) ( e f g ) )
  + ( a ( b c f ) ( d e g ) )
  \\
  &
  - ( a ( b c g ) ( d f e ) )
  - ( a ( b d e ) ( c f g ) )
  - ( a ( b d g ) ( f c e ) )
  + ( a ( b e g ) ( d c f ) )
  - ( a ( b g d ) ( c f e ) )
  \\
  &
  - ( a ( b g d ) ( f e c ) )
  + ( a ( b g f ) ( c e d ) )
  - ( a ( c b d ) ( f e g ) )
  + ( a ( c b e ) ( d f g ) )
  + ( a ( c b e ) ( f g d ) )
  \\
  &
  - ( a ( c d b ) ( e f g ) )
  + ( a ( c g b ) ( f e d ) )
  - ( a ( d b f ) ( c e g ) )
  - ( a ( d b f ) ( e g c ) )
  + ( a ( e g b ) ( d c f ) )
  \\
  &
  \equiv 0.
  \end{align*}
\end{theorem}

\begin{remark}
The following identity with 20 terms and coefficients $\pm 1$ is the simplest identity in degree 7
for the cyclic commutator which increased the rank in the computation described above:
  \allowdisplaybreaks
  \begin{align*}
  &
    ( a b ( c d ( e f g ) ) )
  + ( a b ( c d ( g e f ) ) )
  - ( a b ( e d ( g f c ) ) )
  - ( a b ( g d ( e c f ) ) )
  - ( a b ( f d ( c e g ) ) )
  \\
  &
  - ( a d ( c b ( e g f ) ) )
  - ( a d ( c b ( g f e ) ) )
  + ( a d ( e b ( g c f ) ) )
  + ( a d ( g b ( e f c ) ) )
  + ( a d ( f b ( c g e ) ) )
  \\
  &
  - ( a ( b d c ) ( e g f ) )
  - ( a ( b d c ) ( g f e ) )
  + ( a ( b d e ) ( g c f ) )
  + ( a ( b d g ) ( e f c ) )
  + ( a ( b d f ) ( c g e ) )
  \\
  &
  + ( a ( d b c ) ( e f g ) )
  + ( a ( d b c ) ( g e f ) )
  - ( a ( d b e ) ( g f c ) )
  - ( a ( d b g ) ( e c f ) )
  - ( a ( d b f ) ( c e g ) )
  \\
  &
  \equiv 0.
  \end{align*}
\end{remark}

\begin{theorem}
There are no new identities for the cyclic commutator in degree 9:
every multilinear polynomial identity of degree 9 satisfied by the cyclic commutator is
a consequence of the identities in degrees 3, 5 and 7.
\end{theorem}

\begin{proof}
Owing to the large size of this problem, we use the representation theory of the symmetric group
to decompose the computation into smaller pieces corresponding to the irreducible representations.
A summary of the theory and algorithms underlying this method has been given
by Bremner and Peresi \cite{BremnerPeresi}.  We briefly explain this computation in
the present case; see Table \ref{degree9table}.  A partition will be denoted
  \[
  \lambda = (n_1,\dots,n_k),
  \qquad
  n \ge n_1 \ge \cdots \ge n_k \ge 1,
  \qquad
  n_1 + \cdots + n_k = 9.
  \]
These partitions label the irreducible representations of $S_9$; the dimension of the representation
corresponding to $\lambda$ will be denoted $d_\lambda$.  Given a partition $\lambda$ and a
permutation $\pi \in S_9$, the algorithm of Clifton \cite{Clifton} shows how to efficiently
compute a matrix $A_\lambda(\pi)$.  Furthermore, the formula
  \[
  R_\lambda(\pi) = A_\lambda(1)^{-1} A_\lambda(\pi),
  \]
where 1 is the identity permutation, gives the matrix representing $\pi$ in the representation
corresponding to $\lambda$.

We have already seen in \eqref{lift3} and \eqref{lift5} how to generate, for $n = 3$ and $n = 5$,
the consequences in degree $n+2$ of a ternary identity in degree $n$.
A similar process generates the consequences in degree 9 of a ternary identity in degree 7:
from $K(a,b,\dots,g)$ we perform
($i$) seven substitutions, replacing $x$ by $(x,h,i)$ for $x = a, b, \dots g$, and
($ii$) three multiplications, namely $(K,h,i)$, $(h,K,i)$ and $(h,i,K)$.
In this way we generate all consequences in degree 9 of the ternary Jacobi identity, the (right)
ternary derivation identity, and the 45-term identity of Theorem \ref{degree7theorem};
the total number of these identities is $6 \cdot 8 \cdot 10 + 8 \cdot 10 + 10 = 570$.
Every identity in degree 9, which is satisfied by the cyclic commutator and is a consequence
of the identities of lower degree, is a linear combination of permutations of these 570 identities.
For each representation $\lambda$ of dimension $d = d_\lambda$, we construct a matrix of size
$570 d \times 55 d$ consisting of $d \times d$ blocks.  In the $(i,j)$ block we put the
representation matrix, computed by Clifton's method, of the terms of identity $i$ with association
type $j$.  (Note that we are using all 55 ternary association types in degree 9.)
The rank of this matrix of ``lifted identities'' is denoted ``lifrank'' in Table
\ref{degree9table}.

For each representation $\lambda$ of dimension $d = d_\lambda$, we construct a second matrix of size
$55 d \times 56 d$ consisting of $d \times d$ blocks.  In the $(i,1)$ block we put the
representation matrix of the terms of the expansion in the free associative algebra of the ternary
monomial with association type $i$ and the identity permutation of the variables; in the $(i,i+1)$
block we put the identity matrix; the other blocks are zero.  The rank of this ``expansion matrix'' is always
$55 d$; this number is denoted ``exprank'' in Table \ref{degree9table}.  We compute the row canonical
form of this matrix and identify the rows whose leading 1s occur within the first column of
blocks; the number of these rows is denoted ``toprank''.  The number of remaining rows, whose leading
1s occur to the right of the first column of blocks, is denoted ``allrank''; these rows represent
all the identities satisfied by the cyclic commutator in this representation.

For every representation, we find that ``lifrank = allrank''; every identity in degree 9 satisfied
by the cyclic commutator is a consequence of identities of lower degree.  This completes the proof.
\end{proof}

\begin{definition} \label{definitionNCLTS}
A \textbf{noncommutative Lie triple system} is a vector space $T$ with a trilinear operation
$(-,-,-)\colon T \times T \times T \to T$ satisfying the ternary Jacobi identity, the
(right) ternary derivation identity, and the 45-term identity of Theorem \ref{degree7theorem}.
\end{definition}

An open problem is the existence of special identities for noncommutative Lie
triple systems: polynomial identities satisfied by the cyclic commutator in every associative
algebra, but which do not follow from the identities of Defintion \ref{definitionNCLTS}.

\begin{table}
\[
\begin{array}{rlrrrrr}
  \# & \text{partition} & \text{dimension}
  & \text{lifrank} & \text{exprank} & \text{toprank} & \text{allrank} \\
  \midrule
   1 & 9 &   1 &    55 &    55 &   0 &    55 \\
   2 & 81 &   8 &   435 &   440 &   5 &   435 \\
   3 & 72 &  27 &  1464 &  1485 &  21 &  1464 \\
   4 & 711 &  28 &  1519 &  1540 &  21 &  1519 \\
   5 & 63 &  48 &  2600 &  2640 &  40 &  2600 \\
   6 & 621 & 105 &  5686 &  5775 &  89 &  5686 \\
   7 & 6111 &  56 &  3034 &  3080 &  46 &  3034 \\
   8 & 54 &  42 &  2272 &  2310 &  38 &  2272 \\
   9 & 531 & 162 &  8768 &  8910 & 142 &  8768 \\
  10 & 522 & 120 &  6496 &  6600 & 104 &  6496 \\
  11 & 5211 & 189 & 10231 & 10395 & 164 & 10231 \\
  12 & 51111 &  70 &  3790 &  3850 &  60 &  3790 \\
  13 & 441 &  84 &  4546 &  4620 &  74 &  4546 \\
  14 & 432 & 168 &  9091 &  9240 & 149 &  9091 \\
  15 & 4311 & 216 & 11689 & 11880 & 191 & 11689 \\
  16 & 4221 & 216 & 11689 & 11880 & 191 & 11689 \\
  17 & 42111 & 189 & 10231 & 10395 & 164 & 10231 \\
  18 & 411111 &  56 &  3034 &  3080 &  46 &  3034 \\
  19 & 333 &  42 &  2274 &  2310 &  36 &  2274 \\
  20 & 3321 & 168 &  9091 &  9240 & 149 &  9091 \\
  21 & 33111 & 120 &  6496 &  6600 & 104 &  6496 \\
  22 & 3222 &  84 &  4546 &  4620 &  74 &  4546 \\
  23 & 32211 & 162 &  8768 &  8910 & 142 &  8768 \\
  24 & 321111 & 105 &  5686 &  5775 &  89 &  5686 \\
  25 & 3111111 &  28 &  1519 &  1540 &  21 &  1519 \\
  26 & 22221 &  42 &  2272 &  2310 &  38 &  2272 \\
  27 & 222111 &  48 &  2600 &  2640 &  40 &  2600 \\
  28 & 2211111 &  27 &  1464 &  1485 &  21 &  1464 \\
  29 & 21111111 &   8 &   435 &   440 &   5 &   435 \\
  30 & 111111111 &   1 &    55 &    55 &   0 &    55 \\
\bottomrule
\end{array}
\]
\caption{Ranks of identities in degree 9 for cyclic commutator}
\label{degree9table}
\end{table}

\subsection{Universal associative envelopes}

We can obtain more information about a nonassociative structure by studying its irreducible
finite dimensional representations.  For a structure defined by a multilinear operation, the
first step toward classifying the representations is to construct the universal associative
enveloping algebra.
This generalizes the familiar construction of the universal enveloping algebras of Lie and
Jordan algebras, where an important dichotomy arises: a finite dimensional simple Lie algebra has
an infinite dimensional universal envelope and infinitely many isomorphism classes of
irreducible finite dimensional representations, but a finite dimensional simple Jordan algebra has
a finite dimensional envelope and only finitely many irreducible representations.

The general definition of the universal associative envelope is as follows.
Suppose that $B$ is a subspace, of an associative algebra $A$ over
the field $\mathbb{F}$, closed under the $n$-ary multilinear operation
  \[
  (a_1,\dots,a_n) = \sum_{\sigma \in S_n} x_\sigma a_{\sigma(1)} \cdots a_{\sigma(n)}
  \quad
  (x_\sigma \in \mathbb{F}).
  \]
Write $d = \dim B$ and let $b_1, \dots, b_d$ be a basis of $B$ over
$\mathbb{F}$; we then have the structure constants for the resulting $n$-ary
algebra structure on $B$:
  \[
  (b_{i_1},\dots,b_{i_n}) = \sum_{j=1}^d c_{i_1 \cdots i_n}^j b_j
  \quad
  (1 \le i_1, \dots, i_n \le d).
  \]
Let $F\langle B \rangle$ be the free associative algebra generated by the symbols $b_1, \dots, b_d$
(this ambiguity should not cause confusion). Consider the ideal $I \subseteq F\langle B \rangle$ generated by
the $d^n$ elements
  \[
  \sum_{\sigma \in S_n} x_\sigma b_{i_{\sigma(1)}} \cdots b_{i_{\sigma(n)}}
  -
  \sum_{j=1}^d c_{i_1 \cdots i_n}^j b_j
  \quad
  (1 \le i_1, \dots, i_n \le d).
  \]
The quotient algebra $U(B) = F\langle B \rangle / I$ is the universal
associative enveloping algebra of the $n$-ary structure on $B$; by assumption,
the natural map $B \to U(B)$ will be injective, since the $n$-ary structure on
$B$ is defined in terms of the associative structure on $A$. This generalizes
the construction of the enveloping algebras of Lie algebras, where $I$ is
generated by the elements $b_i b_j - b_j b_i - [b_i,b_j]$, and of
Jordan algebras, where $I$ is generated by $b_i b_j + b_j b_i - b_i \circ b_j$.
If $B$ is a finite-dimensional Lie (resp.~Jordan) algebra, then $U(B)$
is infinite-dimensional (resp.~finite-dimensional).

More generally, the same construction applies to any $n$-ary algebra which
satisfies the same low-degree polynomial identities as the $n$-ary operation
$(a_1,\dots,a_n)$.  This gives rise to a universal associative enveloping
algebra; however, the natural map $B \to U(B)$ is no longer necessarily
injective: for example, the universal enveloping algebra of an exceptional
Jordan algebra.  Once a set of generators for the ideal $I$ is known, one can
compute a noncommutative Gr\"obner basis for this ideal, and then use this to
obtain a basis and structure constants for $U(B)$.

\subsection{An example}

We make the vector space of $n \times n$ matrices of trace 0 into a ternary
algebra $C_n$ with the cyclic commutator $\omega(x,y,z) = xyz - yzx$ as the trilinear operation.
In the simplest case, $n = 2$, we have this basis for $C_2$:
  \[
  a = \left[ \begin{array}{rr} 1 & 0 \\ 0 & -1 \end{array} \right],
  \qquad
  b = \left[ \begin{array}{rr} 0 & 1 \\ 0 &  0 \end{array} \right],
  \qquad
  c = \left[ \begin{array}{rr} 0 & 0 \\ 1 &  0 \end{array} \right].
  \]
The universal associative envelope $U(C_2)$ is the quotient of the free associative algebra
with three generators (also denoted $a, b, c$) modulo the ideal generated by the elements
$xyz - yzx - \omega(x,y,z)$ for $x, y, z \in \{ a, b, c \}$.  This gives the following set
of 24 ideal generators, in reverse degree lexicographical order:
  \begin{equation}
  \label{generators}
  \left\{
  \begin{array}{llll}
  c^2b - bc^2, &\quad
  c^2b - cbc + c, &\quad
  c^2a - ac^2, &\quad
  c^2a - cac, \\
  cbc - bc^2 - c, &\quad
  cb^2 - bcb + b, &\quad
  cb^2 - b^2c, &\quad
  cba - acb, \\
  cba - bac - a, &\quad
  cac - ac^2, &\quad
  cab - bca + a, &\quad
  cab - abc + a, \\
  ca^2 - aca -2 c, &\quad
  ca^2 - a^2c, &\quad
  bcb - b^2c - b, &\quad
  bca - abc, \\
  b^2a - ab^2, &\quad
  b^2a - bab, &\quad
  bac - acb + a, &\quad
  bab - ab^2, \\
  ba^2 - aba -2 b, &\quad
  ba^2 - a^2b, &\quad
  aca - a^2c +2 c, &\quad
  aba - a^2b +2 b.
  \end{array}
  \right.
  \end{equation}
We compute a Gr\"obner basis for this ideal following the ideas of Bergman \cite{Bergman}
and the exposition by de Graaf \cite{deGraaf}.
We self-reduce the set of generators \eqref{generators} by performing noncommutative division with remainder
in order to eliminate terms which contain leading monomials of other terms.  This leaves a
set of 16 ideal generators:
  \begin{equation}
  \label{reduced}
  \left\{
  \begin{array}{llll}
  c^2b - bc^2, &\quad
  c^2a - ac^2, &\quad
  cbc - bc^2 - c, &\quad
  cb^2 - b^2c, \\
  cba - acb, &\quad
  cac - ac^2, &\quad
  cab - abc + a, &\quad
  ca^2 - a^2c, \\
  bcb - b^2c - b, &\quad
  bca - abc, &\quad
  b^2a - ab^2, &\quad
  bac - acb + a, \\
  bab - ab^2, &\quad
  ba^2 - a^2b, &\quad
  aca - a^2c +2 c, &\quad
  aba - a^2b +2 b.
  \end{array}
  \right.
  \end{equation}
We find all compositions of these generators, obtaining 93 elements, and then
compute the normal forms of the compositions by reducing them modulo the ideal generators;
we obtain 18 elements which must be included as new ideal generators:
  \begin{equation}
  \label{compositions}
  \left\{
  \begin{array}{llll}
  \multicolumn{2}{l}{a^3cb - a^3bc +2 abc + a^3 -2 a,} &
  \multicolumn{2}{l}{a^3cb - a^3bc -2 acb - a^3 +2 a,}
  \\
  \multicolumn{2}{l}{a^2cb - a^2bc -2 cb + a^2,} &
  \multicolumn{2}{l}{a^2cb - a^2bc +2 bc - a^2,}
  \\
  c^3, &\quad
  bc^2, &
  b^2c, &\quad
  b^2c - a^2b + b,
  \\
  b^3, &\quad
  ac^2, &
  acb + abc - a, &\quad
  ab^2,
  \\
  a^2c - c, &\quad
  a^2b - b, &
  c^2, &\quad
  ca + ac,
  \\
  b^2, &\quad
  ba + ab.
  \end{array}
  \right.
  \end{equation}
We combine the generators \eqref{reduced} with the compositions \eqref{compositions},
and self-reduce the resulting set, obtaining a new set of 8 ideal generators:
  \begin{equation}
  \label{newgenerators}
  a^2c - c, \;\;
  a^2b - b, \;\;
  a^3 - a, \;\;
  c^2, \;\;
  cb + bc - a^2, \;\;
  ca + ac, \;\;
  b^2, \;\;
  ba + ab.
  \end{equation}
We repeat the same process once more: finding all compositions of the generators,
and computing the normal forms of the compositions modulo the generators.
Every composition reduces to 0, and hence \eqref{newgenerators} is a Gr\"obner basis
for the ideal.
From this we easily obtain a vector space basis for
the universal envelope $U(C_2)$: the cosets of the monomials which do not contain
the leading monomial of any element of the Gr\"obner basis.  Hence $U(C_2)$
is finite dimensional and has this basis:
  \begin{equation}
  \label{ubasis}
  1, \qquad a, \qquad b, \qquad c, \qquad a^2, \qquad ab, \qquad ac, \qquad bc, \qquad abc.
  \end{equation}
The multiplication for this monomial basis is given in Table \ref{UC2table}, where we
write monomials but mean cosets.  If the product of two basis monomials is not a basis
monomial, then we must compute its normal form modulo the Gr\"obner basis.

  \begin{table}
  \[
  \begin{array}{c|ccccccccc}
  \cdot &
  1 &
  a &
  b &
  c &
  a^2 &
  ab &
  ac &
  bc &
  abc \\ \midrule
  1 &
  1 &
  a &
  b &
  c &
  a^2 &
  ab &
  ac &
  bc &
  abc \\
  a &
  a &
  a^2 &
  ab &
  ac &
  a &
  b &
  c &
  abc &
  bc \\
  b &
  b &
  - ab &
  0 &
  bc &
  b &
  0 &
  - abc &
  0 &
  0 \\
  c &
  c &
  - ac &
  a^2 - bc &
  0 &
  c &
  - a  + abc &
  0 &
  c &
  - ac \\
  a^2 &
  a^2 &
  a &
  b &
  c &
  a^2 &
  ab &
  ac &
  bc &
  abc \\
  ab &
  ab &
  - b &
  0 &
  abc &
  ab &
  0 &
  - bc &
  0 &
  0 \\
  ac &
  ac &
  - c &
  a - abc &
  0 &
  ac &
  - a^2  + bc &
  0 &
  ac &
  - c \\
  bc &
  bc &
  abc &
  b &
  0 &
  bc &
  ab &
  0 &
  bc &
  abc \\
  abc &
  abc &
  bc &
  ab &
  0 &
  abc &
  b &
  0 &
  abc &
  bc
  \end{array}
  \]
  \smallskip
  \caption{Multiplication table for monomial basis of $U(C_2)$}
  \label{UC2table}
  \end{table}

We now compute the Wedderburn decomposition of $U(C_2)$ using the algorithms
described in the author's survey paper \cite{BremnerW}.
The radical of $U(C_2)$ consists of the elements whose coefficient vectors with respect to
the ordered basis \eqref{ubasis} belong to the nullspace of the Dickson matrix (Table \ref{dickson}),
but this matrix has full rank.
It follows that $U(C_2)$ is semisimple, and hence a direct sum of full matrix algebras.

  \begin{table}
  \[
  \left[
  \begin{array}{rrrrrrrrr}
  9 & 0 & 0 & 0 & 8 &  0 &  0 & 4 & 0 \\
  0 & 8 & 0 & 0 & 0 &  0 &  0 & 0 & 4 \\
  0 & 0 & 0 & 4 & 0 &  0 &  0 & 0 & 0 \\
  0 & 0 & 4 & 0 & 0 &  0 &  0 & 0 & 0 \\
  8 & 0 & 0 & 0 & 8 &  0 &  0 & 4 & 0 \\
  0 & 0 & 0 & 0 & 0 &  0 & -4 & 0 & 0 \\
  0 & 0 & 0 & 0 & 0 & -4 &  0 & 0 & 0 \\
  4 & 0 & 0 & 0 & 4 &  0 &  0 & 4 & 0 \\
  0 & 4 & 0 & 0 & 0 &  0 &  0 & 0 & 4
  \end{array}
  \right]
  \]
  \smallskip
  \caption{Dickson matrix for $U(C_2)$}
  \label{dickson}
  \end{table}

  \begin{table}
  \begin{align*}
  M
  =
  \frac12
  &\left[
  \begin{array}{rrrrrrrrr}
   2 &  0 & 0 &  0 &  0 &  0 &  0 &  0 & 0 \\
   0 &  1 & 0 &  0 &  0 & -1 &  0 &  0 & 0 \\
   0 &  0 & 0 &  1 &  0 &  0 &  0 & -1 & 0 \\
   0 &  0 & 1 &  0 &  0 &  0 & -1 &  0 & 0 \\
  -2 &  1 & 0 &  0 &  0 &  1 &  0 &  0 & 0 \\
   0 &  0 & 0 & -1 &  0 &  0 &  0 & -1 & 0 \\
   0 &  0 & 1 &  0 &  0 &  0 &  1 &  0 & 0 \\
   0 & -1 & 0 &  0 &  1 & -1 &  0 &  0 & 1 \\
   0 & -1 & 0 &  0 & -1 &  1 &  0 &  0 & 1
  \end{array}
  \right]
  \end{align*}
  \caption{The change of basis matrix for $U(C_2)$}
  \label{Mtable}
  \end{table}

A basis for the center of $U(C_2)$ is easily found, and consists of these three elements:
  \[
  1, \qquad
  a-2abc, \qquad
  a^2.
  \]
From this we obtain a basis of orthogonal primitive idempotents for the center:
  \[
  1-a^2, \qquad
  \tfrac12 a + \tfrac12 a^2 - abc, \qquad
  -\tfrac12 a + \tfrac12 a^2 + abc.
  \]
The first idempotent generates a 1-dimensional ideal, and the second and third
each generate a 4-dimensional ideal.
From this it follows that
  \begin{equation}
  \label{wedderburn}
  U(C_2) \approx \mathbb{F} \oplus M_2(\mathbb{F}) \oplus M_2(\mathbb{F}).
  \end{equation}
Hence $C_2$ has exactly three distinct irreducible finite dimensional representations:
the 1-dimensional trivial representation, the 2-dimensional natural representation, and
another 2-dimensional representation which is in fact the negative of the natural representation.
Moreover, $U(C_2)$ satisfies the standard identity for $2 \times 2$ matrices.
From this point of view $C_2$ is more like a Jordan structure than a Lie structure.

Further calculations give the matrix units in the 4-dimensional ideals, and so we obtain
another basis for $U(C_2)$:
  \[
  E_{11}^{(1)}, \quad
  E_{11}^{(2)}, \quad E_{12}^{(2)}, \quad E_{21}^{(2)}, \quad E_{22}^{(2)}, \quad
  E_{11}^{(3)}, \quad E_{12}^{(3)}, \quad E_{21}^{(3)}, \quad E_{22}^{(3)}.
  \]
The columns of the matrix $M$ in Table \ref{Mtable} give the coefficients of these basis elements in terms
of the original basis elements \eqref{ubasis}.  The inverse matrix $M^{-1}$ gives the coefficients of the original basis
in terms of the matrix units, and from the columns of the inverse we extract the representation matrices.

\begin{remark}
The theory of noncommutative Gr\"obner bases has been extended recently to associative dialgebras
by Bokut et al.~\cite{BCL}.  An interesting open problem is to use these results to construct the
universal associative enveloping dialgebras of certain finite dimensional nonassociative dialgebras.
In the case that the enveloping dialgebra is finite dimensional, then it would be useful to have a
generalization to dialgebras of the classical Wedderburn structure theory for associative algebras.
A first step in this direction has been taken recently by M\'artin-Gonz\'alez \cite{Martin}.
\end{remark}

\subsection{Dialgebra analogues of the cyclic commutator}

Applying the KP algorithm to the ternary Jacobi identity gives three identities,
each of which is equivalent to the following identity relating the three new operations:
  \[
  (a,b,c)_3 + (b,c,a)_2 + (c,a,b)_1 \equiv 0.
  \]
Hence the third new operation can be eliminated using the equation
  \[
  (a,b,c)_3 \equiv {} - (c,a,b)_1 - (b,c,a)_2.
  \]
Applying the KP algorithm to the ternary Jacobi identity gives five identities:
  \allowdisplaybreaks
  \begin{align*}
  (a,b,(c,d,e)_1)_1 &\equiv ((a,b,c)_1,d,e)_1 + (c,(a,b,d)_1,e)_2 + (c,d,(a,b,e)_1)_3, \\
  (a,b,(c,d,e)_1)_2 &\equiv ((a,b,c)_2,d,e)_1 + (c,(a,b,d)_2,e)_2 + (c,d,(a,b,e)_2)_3, \\
  (a,b,(c,d,e)_1)_3 &\equiv ((a,b,c)_3,d,e)_1 + (c,(a,b,d)_1,e)_1 + (c,d,(a,b,e)_1)_1, \\
  (a,b,(c,d,e)_2)_3 &\equiv ((a,b,c)_3,d,e)_2 + (c,(a,b,d)_3,e)_2 + (c,d,(a,b,e)_1)_2, \\
  (a,b,(c,d,e)_3)_3 &\equiv ((a,b,c)_3,d,e)_3 + (c,(a,b,d)_3,e)_3 + (c,d,(a,b,e)_3)_3.
  \end{align*}
Eliminating the third operation gives five identities relating the first two operations:
  \allowdisplaybreaks
  \begin{align*}
  &
  (a,b,(c,d,e)_1)_1 - ((a,b,c)_1,d,e)_1 - (c,(a,b,d)_1,e)_2 + ((a,b,e)_1,c,d)_1
  \\
  &\quad
  + (d,(a,b,e)_1,c)_2
  \equiv 0,
  \\
  &
  (a,b,(c,d,e)_1)_2 - ((a,b,c)_2,d,e)_1 - (c,(a,b,d)_2,e)_2 + ((a,b,e)_2,c,d)_1
  \\
  &\quad
  + (d,(a,b,e)_2,c)_2
  \equiv 0,
  \\
  &
  ((c,d,e)_1,a,b)_1 + (b,(c,d,e)_1,a)_2 - ((c,a,b)_1,d,e)_1 - ((b,c,a)_2,d,e)_1
  \\
  &\quad
  + (c,(a,b,d)_1,e)_1 + (c,d,(a,b,e)_1)_1
  \equiv 0,
  \\
  &
  ((c,d,e)_2,a,b)_1 + (b,(c,d,e)_2,a)_2 - ((c,a,b)_1,d,e)_2 - ((b,c,a)_2,d,e)_2
  \\
  &\quad
  - (c,(d,a,b)_1,e)_2 - (c,(b,d,a)_2,e)_2 + (c,d,(a,b,e)_1)_2
  \equiv 0,
  \\
  &
  ((e,c,d)_1,a,b)_1 + ((d,e,c)_2,a,b)_1 + (b,(e,c,d)_1,a)_2 + (b,(d,e,c)_2,a)_2
  \\
  &\quad
  - (e,(c,a,b)_1,d)_1 - (e,(b,c,a)_2,d)_1 - (d,e,(c,a,b)_1)_2 - (d,e,(b,c,a)_2)_2
  \\
  &\quad
  - (e,c,(d,a,b)_1)_1 - (e,c,(b,d,a)_2)_1 - ((d,a,b)_1,e,c)_2 - ((b,d,a)_2,e,c)_2
  \\
  &\quad
  - ((e,a,b)_1,c,d)_1 - ((b,e,a)_2,c,d)_1 - (d,(e,a,b)_1,c)_2 - (d,(b,e,a)_2,c)_2
  \equiv 0.
  \end{align*}
Applying the KP algorithm to the 45-term identity of Theorem \ref{degree7theorem}
will produce seven identities from which we can eliminate the third operation.
All these identities together will define the dialgebra analogue of noncommutative
Lie triple systems.

If we apply the BSO algorithm to the cyclic commutator, $\omega(a,b,c) = abc - bca$,
then we obtain these three dialgebra operations:
  \[
  \widehat{\omega}_1(a,b,c) = \widehat{a}bc - bc\widehat{a},
  \qquad
  \widehat{\omega}_2(a,b,c) = a\widehat{b}c - \widehat{b}ca,
  \qquad
  \widehat{\omega}_3(a,b,c) = ab\widehat{c} - b\widehat{c}a.
  \]
We have the linear dependence relation,
  \[
  \widehat{\omega}_1(a,b,c) + \widehat{\omega}_2(c,a,b) + \widehat{\omega}_3(b,c,a) = 0,
  \]
so we only retain
$\widehat{\omega}_1(a,b,c)$ and $\widehat{\omega}_2(a,b,c)$.  It is an open problem to
determine the polynomial identities of degrees 3, 5 and 7 satisfied by these operations
in every associative dialgebra, and to check whether these identities are equivalent to
those produced by the KP algorithm.


\section{Conjecture relating the KP and BSO algorithms}

In this section we state a conjecture first formulated by Bremner, Felipe, and S\'anchez-Ortega \cite{BFSO}.
Let $\mathbb{F}$ be a field, and let $\omega$ be a multilinear $n$-ary operation over $\mathbb{F}$.
Fix a degree $d$ and consider the multilinear polynomial identities of degree $e \le d$ satisfied by $\omega$.
Precisely, let $A_e$ be the multilinear subspace of degree $e$ in the free nonassociative $n$-ary algebra
on $e$ generators.
Let $I_e \subseteq A_e$ be the subspace of polynomials which vanish identically when the $n$-ary operation
is replaced by $\omega$.
The multilinear identities of degree $\le d$ satisfied by $\omega$ are then
  \[
  I_d(\omega) = \bigoplus_{1 \le e \le d} I_e.
  \]
Applying the KP algorithm to the identities in $I_d(\omega)$ produces
multilinear identities for $n$ new $n$-ary operations.
Precisely, let $B_e$ be the multilinear subspace of degree $e$ in
the free nonassociative algebra with $n$ operations of arity $n$.
Let KP$(I_e) \subseteq B_e$ be the subspace obtained by applying the KP algorithm to $I_e$,
and define
  \[
  \mathrm{KP}_d(\omega) = \bigoplus_{1 \le e \le d} \mathrm{KP}(I_e).
  \]
We now consider a different path to the same goal.
Applying the BSO algorithm to $\omega$ produces $n$ multilinear $n$-ary operations
$\widehat{\omega}_1, \dots, \widehat{\omega}_n$.
Consider the multilinear polynomial identities of degree $e \le d$ satisfied by
$\widehat{\omega}_1, \dots, \widehat{\omega}_n$.
Precisely, let $J_e \subseteq B_e$ be the subspace of polynomials which vanish identically
when the $n$ operations are replaced by $\widehat{\omega}_1, \dots, \widehat{\omega}_n$ and define
  \[
  J_d(\widehat{\omega}_1,\dots,\widehat{\omega}_n) = \bigoplus_{1 \le e \le d} J_e.
  \]

\begin{conjecture} \label{mainconjecture}
If $\mathbb{F}$ has characteristic 0 or $p > d$ then
  \[
  \mathrm{KP}_d(\omega) = J_d(\widehat{\omega}_1,\dots,\widehat{\omega}_n).
  \]
\end{conjecture}

The conjecture states that these two processes give the same results when
the group algebra $\mathbb{F} S_d$ is semisimple:
  \begin{itemize}
  \item
  Find the multilinear polynomial identities satisfied by $\omega$, and apply the KP algorithm.
  \item
  Apply the BSO algorithm, and find the multilinear polynomial identities satisfied by
  $\widehat{\omega}_1,\dots,\widehat{\omega}_n$.
  \end{itemize}
The conjecture is equivalent to the commutativity of this diagram:
\[
\begin{CD}
\omega @>\text{BSO}>> \widehat{\omega}_1,\dots,\widehat{\omega}_n
\\
@VVV @VVV
\\
I_d(\omega) @>\text{KP}>>
\begin{array}{c}
J_d(\widehat{\omega}_1,\dots,\widehat{\omega}_n)
\\
\stackrel{?}{=} \; \mathrm{KP}_d(\omega)
\\
\end{array}
\end{CD}
\]
The vertical arrows indicate the process of determining the multilinear polynomial identities satisfied
by the given operations.

\begin{remark}
Significant progress toward a proof of this conjecture has recently been announced by Kolesnikov and Voronin
\cite{KolesnikovVoronin}.
\end{remark}


\section{Open Problems}

In this final section we list some open problems related to generalizing
well-known varieties of algebras to the setting of dialgebras.

The next step beyond Lie and Malcev algebras leads to the notion of Bol algebras.
Just as Lie algebras (respectively Malcev algebras) can be defined by the polynomial
identities satisfied by the commutator in every associative algebra (respectively
alternative algebra), so also Bol algebras can be defined by the identities
satisfied by the commutator and associator in every right alternative algebra;
see P\'erez-Izquierdo \cite{PIBol}, Hentzel and Peresi \cite{HP}.  One can find
the defining identities for right alternative dialgebras by an application of the
KP algorithm, and then use computer algebra to determine the identities satisfied
by the dicommutator and the left, right, and inner associators in every right
alternative dialgebra.  On the other hand, one can apply the KP algorithm to the
defining identities for Bol algebras.  Are these two sets of identities equivalent?

Beyond Bol algebras, one obtains structures with binary, ternary and quaternary
operations, which are closely related to the tangent algebras of monoassociative loops;
see Bremner and Madariaga \cite{BM}.  These structures can be defined by the
identities satisfied by the commutator, associator and quaternator
$\langle a,b,c,d \rangle = (ab,c,d) - (a,c,d)b - a(b,c,d)$
in every power associative algebra.  What is the dialgebra analogue of these structures?

The tangent algebras of analytic loops have binary and ternary operations, which
correspond to the commutator and associator in a free nonassociative algebra;
these operations are related by the Akivis identity:
  \allowdisplaybreaks
  \begin{align*}
  &
  [[a,b],c] + [[b,c],a] + [[c,a],b] \equiv
  \\
  &
  (a,b,c) - (a,c,b) - (b,a,c) + (b,c,a) + (c,a,b) + (c,b,a).
  \end{align*}
To obtain the correct generalization of Lie's third theorem to an arbitrary
analytic loop, one must consider the infinite family of multilinear operations
whose polynomial identities define the variety of Sabinin algebras.  The basic
references on Akivis and Sabinin algebras are P\'erez-Izquierdo \cite{PISabinin},
Shestakov and Umirbaev \cite{SU}.
What can one say about Akivis dialgebras and Sabinin dialgebras?

In a different direction, a generalization of dialgebras to structures with three associative operations
has been considered by Loday and Ronco \cite{LodayRonco}; see also Casas \cite{Casas}.
It would be interesting to generalize the KP algorithm to the setting of trialgebras:
that is, for any variety of nonassociative structures, give a functorial definition of
the corresponding variety of trialgebras.  For the case of binary operations, see recent work of 
Gubarev and Kolesnikov \cite{GubarevKolesnikov2}.
One can also consider the application of the KP algorithm
to the variety of associative dialgebras: this would produce a variety of structures
with four associative operations satisfying various identities.
This procedure can clearly be iterated
$n$ times to produce a variety of structures with $2^n$ associative operations related
by certain natural identities.


\end{document}